\documentclass[a4paper]{article}

\usepackage{amsmath,amssymb,amsfonts,amsthm}
\usepackage[margin=2.54cm]{geometry}

\newtheorem{theorem}{Theorem}[section]
\newtheorem{lemma}[theorem]{Lemma}
\newtheorem{proposition}[theorem]{Proposition}
\newtheorem{corollary}[theorem]{Corollary}
\newtheorem{hypothesis}[theorem]{Hypothesis}
\numberwithin{equation}{section}

\theoremstyle{remark}
\newtheorem{remark}[theorem]{Remark}
\newtheorem{example}[theorem]{Example}

\theoremstyle{definition}
\newtheorem{definition}[theorem]{Definition}

\newcommand{\Ric}{\mathop{\mathrm{Ric}}\nolimits}
\newcommand{\Ad}{\mathop{\mathrm{Ad}}\nolimits}
\newcommand{\ad}{\mathop{\mathrm{ad}}\nolimits}
\newcommand{\tr}{\mathop{\mathrm{tr}}\nolimits}

\author{Mark Gould\thanks{School of Mathematics and Physics, The
University of Queensland, St Lucia,~QLD 4072, Australia}~\thanks{Mark Gould's research is supported under Australian Research Council's Discovery Projects funding scheme (DP140101492 and~DP160101376).} \\
\small{\texttt{m.gould1@uq.edu.au}} \and Artem Pulemotov\footnotemark[1]~\thanks{Artem Pulemotov is a recipient of the Australian Research Council Discovery Early-Career Researcher Award (DE150101548).} \\
\small{\texttt{a.pulemotov@uq.edu.au}}}

\title{The prescribed Ricci curvature problem on homogeneous spaces with intermediate subgroups}

\begin{document}

\maketitle

\begin{abstract}
Consider a compact Lie group $G$ and a closed subgroup $H<G$. Suppose $\mathcal M$ is the set of $G$-invariant Riemannian metrics on the homogeneous space~$M=G/H$. We obtain a sufficient condition for the existence of $g\in\mathcal M$ and $c>0$ such that the Ricci curvature of $g$ equals $cT$ for a given~$T\in\mathcal M$. This condition is also necessary if the isotropy representation of $M$ splits into two inequivalent irreducible summands. Immediate and potential applications include new existence results for Ricci iterations.
\end{abstract}

\tableofcontents

\section{Introduction}\label{sec_intro}

Consider a smooth manifold $M$ and a symmetric (0,2)-tensor field $T$ on~$M$. The prescribed Ricci curvature problem consists in finding a Riemannian metric $g$ such that
\begin{align}\label{eq_basic_PRC}
\Ric g=T,
\end{align}
where $\Ric g$ denotes the Ricci curvature of $g$. The investigation of this problem is an important segment of geometric analysis with strong ties to flows and relativity. While many mathematicians have made significant contributions to the study of~\eqref{eq_basic_PRC}, a particularly large amount of work was done by D.~DeTurck and his collaborators. The reader will find surveys in~\cite[Chapter~5]{AB87} and~\cite[Section~6.5]{TA98}. For more recent results, see~\cite{AP13,AP16,ED17a,ED17b} and references therein.

Suppose the manifold $M$ is closed and the tensor field $T$ is positive-definite. It is possible for equation~\eqref{eq_basic_PRC} to have no solutions. Moreover, in a number of settings, a metric $g$ such that
\begin{align}\label{eq_PRC}
\Ric g=cT
\end{align}
only exists for one value of $c\in\mathbb R$; see, e.g.,~\cite{RH84,AP16}. This observation suggests a change of paradigm in the study of the prescribed Ricci curvature problem. Namely, instead of trying to solve~\eqref{eq_basic_PRC}, one should search for a metric $g$ and a constant $c>0$ satisfying~\eqref{eq_PRC}. The idea of shifting focus from~\eqref{eq_basic_PRC} to~\eqref{eq_PRC} dates back to R.~Hamilton's work~\cite{RH84} and D.~DeTurck's work~\cite{DDT85}. Note that such a shift may be unreasonable on an open manifold or a manifold with non-empty boundary.

In the paper~\cite{AP16}, the second-named author initiated the investigation of equation~\eqref{eq_PRC} on homogeneous spaces. More precisely, consider a compact connected Lie group $G$ and a closed connected subgroup $H<G$. Let $M$ be the homogeneous space~$G/H$. We denote by $\mathcal M$ the set of $G$-invariant Riemannian metrics on~$M$ and assume the tensor field $T$ lies in~$\mathcal M$. The main theorem of~\cite{AP16} states that a metric $g\in\mathcal M$ and a constant $c>0$ satisfying~\eqref{eq_PRC} can be found if $H$ is a maximal connected Lie subgroup of~$G$. Further results in~\cite{AP16} address the prescribed Ricci curvature problem on $M$ in the case where the isotropy representation of $M$ splits into two inequivalent irreducible summands. The reader will find a classification of homogeneous spaces possessing this property in~\cite{WDMK08,CH12}. Several authors have studied their geometry in detail; see, e.g.,~\cite{AAIC11,MB14,APYRsubm}.

The main result of the present paper, Theorem~\ref{thm_PRC}, provides a sufficient condition for the existence of $g\in\mathcal M$ and $c>0$ satisfying~\eqref{eq_PRC} in the case where the maximality assumption on $H$ does not hold. This condition is, in fact, necessary when the isotropy representation of $M$ splits into two inequivalent irreducible summands. To describe the result further, assume that $\mathfrak g$ and $\mathfrak h$ are the Lie algebras of $G$ and~$H$. As before, we demand that $T$ lie in~$\mathcal M$. Imposing natural requirements on the Lie subalgebras of $\mathfrak g$ that contain~$\mathfrak h$, we show that the existence of $g\in\mathcal M$ and $c>0$ satisfying~\eqref{eq_PRC} is guaranteed by an array of simple inequalities for~$T$.

Theorem~\ref{thm_PRC} applies on a broad class of homogeneous spaces. For instance, its assumptions hold if $M$ is a generalised flag manifold. Previous literature provides little information concerning the solvability of~\eqref{eq_PRC} on such manifolds. However, several other aspects of their geometry have been investigated thoroughly; see the survey~\cite{AA15}.

As far as applications are concerned, Theorem~\ref{thm_PRC} leads to new existence results for Ricci iterations. More precisely, consider a sequence $(g_i)_{i=1}^\infty$ of Riemannian metrics on a smooth manifold. One calls $(g_i)_{i=1}^\infty$ a \emph{Ricci iteration} if
\begin{align}\label{eq_iter}
\Ric g_i=g_{i-1}
\end{align}
for $i\in\mathbb N\setminus\{1\}$. Introduced by Y.~Rubinstein in~\cite{YR07}, sequences satisfying~\eqref{eq_iter} have been investigated intensively in the framework of K\"ahler geometry; see the survey~\cite{YR14}. The study of such sequences on homogeneous spaces was initiated in~\cite{APYRsubm}. There are close connections between~\eqref{eq_iter} and the Ricci flow. Some of these connections are explained in~\cite[Section~6]{YR14} and~\cite[Subsection~2.2]{APYRsubm}.

In the present paper, we obtain a new existence result for Ricci iterations by exploiting one of the corollaries of Theorem~\ref{thm_PRC}. The assumptions of this result appear to be quite restrictive, and examples of homogeneous spaces to which it applies are scarce. However, we anticipate that Theorem~\ref{thm_PRC} and the underlying techniques will lead to substantial further advances in the study of Ricci iterations in the future.

It is interesting to place our analysis of~\eqref{eq_PRC} into the context of the theory of homogeneous Einstein metrics. We refer to~\cite[Chapter~7]{AB87} for an introduction to this theory and some foundational results. The surveys~\cite{MW99,YNERVS07,MW12,AA15} contain overviews of more recent work. According to~\cite[Theorem~(2.2)]{MWWZ86}, a metric $g\in\mathcal M$ satisfying the Einstein equation
\begin{align}\label{eq_Einstein_intro}
\Ric g=\lambda g
\end{align}
for some $\lambda\in\mathbb R$ exists if $H$ is a maximal connected Lie subgroup of~$G$. Whether such $g\in\mathcal M$ can be found when this assumption does not hold is a difficult question. The papers~\cite{CB04,CBMWWZ04} offer several sufficient conditions for the answer to be positive, while~\cite[\S3]{MWWZ86} discusses a situation in which the answer is negative.

One observes a number of similarities \emph{and} differences between the analytical properties of~\eqref{eq_PRC} and those of~\eqref{eq_Einstein_intro} on homogeneous spaces. As shown in~\cite{AP16}, a metric $g\in\mathcal M$ satisfies~\eqref{eq_PRC} for some $c\in\mathbb R$ if and only if it is a critical point of the scalar curvature functional $S$ on the set
\begin{align}\label{def_MT}
\mathcal M_T=\{g\in\mathcal M\,|\,\tr_gT=1\},
\end{align}
where $\tr_gT$ denotes the trace of $T$ with respect to~$g$. Under the assumptions of Theorem~\ref{thm_PRC}, $S$~has a global maximum on~$\mathcal M_T$. Correspondingly, it is well-known that $g\in\mathcal M$ satisfies~\eqref{eq_Einstein_intro} if and only if it is a critical point of $S$ on the set 
\begin{align}\label{intro_def_M1}
\mathcal M_1=\{g\in\mathcal M\,|\,M~\mbox{has volume~1 with respect to}~g\}. 
\end{align}
This fact underlies the proofs of the main results of~\cite{MWWZ86,CB04,CBMWWZ04}. However, according to~\cite[Theorem~(2.4)]{MWWZ86} and~\cite[Theorem~1.2]{CB04}, it is only in very special situations that $S$ can have a global maximum on~$\mathcal M_1$.

The paper is organised as follows. In Section~\ref{sec_prelim}, we state and prove our main result, Theorem~\ref{thm_PRC}. We also present two corollaries, one of which will be essential to our study of Ricci iterations. Section~\ref{sec_2sum} explores equation~\eqref{eq_PRC} on homogeneous spaces with two inequivalent irreducible isotropy summands. We demonstrate, by appealing to~\cite[Proposition~3.1]{AP16}, that Theorem~\ref{thm_PRC} is optimal in this setting. Section~\ref{sec_flag} discusses the application of our results on generalised flag manifolds. As a specific example, we consider the space~$G_2/U(2)$ with $U(2)$ corresponding to the long root of~$G_2$. This space has three pairwise inequivalent irreducible summands in its isotropy representation. Finally, Section~\ref{sec_iter} deals with the existence of Ricci iterations.

Most of the results of the present paper, including Theorem~\ref{thm_PRC}, are announced in~\cite{MGAP17}.

\section{The existence of metrics with prescribed Ricci curvature}\label{sec_prelim}

As in Section~\ref{sec_intro}, we consider a compact connected Lie group $G$ and a closed connected subgroup $H<G$. Assume the homogeneous space $M=G/H$ has dimension~3 or higher, i.e.,
\begin{align}\label{dimM}
\dim M=n\ge3. 
\end{align}
Choose a scalar product $Q$ on $\mathfrak g$ induced by a bi-invariant Riemannian metric on $G$. If $\mathfrak u$ and $\mathfrak v$ are subspaces of $\mathfrak g$ such that $\mathfrak u\subset\mathfrak v$, we use the notation $\mathfrak v\ominus\mathfrak u$ for the $Q$-orthogonal complement of $\mathfrak u$ in $\mathfrak v$. Define
\begin{align*}
\mathfrak m=\mathfrak g\ominus\mathfrak h.
\end{align*}
It is clear that $\mathfrak m$ is $\Ad(H)$-invariant. The representation $\Ad(H)|_{\mathfrak m}$ is equivalent to the isotropy representation of $G/H$. We standardly identify $\mathfrak m$ with the tangent space $T_HM$.

\subsection{Preliminaries}\label{subsec_nonstand_nota}

The space $\mathcal M$ of $G$-invariant Riemannian metrics on $M$ carries a natural smooth manifold structure; see, e.g.,~\cite[pages~6318--6319]{YNERVS07}. The properties of this space are discussed in~\cite[Subsection~4.1]{CB04} in great detail. In what follows, we implicitly identify $g\in\mathcal M$ with the bilinear form induced by $g$ on $\mathfrak m$ via the identification of $T_HM$ and $\mathfrak m$. The scalar curvature $S(g)$ of a metric $g\in\mathcal M$ is constant on $M$. Therefore, we may interpret $S(g)$ as the result of applying a functional $S:\mathcal M\to\mathbb R$ to $g\in\mathcal M$. Standard formulas for the scalar curvature (see, e.g.,~\cite[Corollary~7.39]{AB87}) imply that $S$ is differentiable on $\mathcal M$. Given $T\in\mathcal M$, the space $\mathcal M_T$ defined by~\eqref{def_MT} has a smooth manifold structure inherited from~$\mathcal M$.

The following result is a special case of~\cite[Lemma~2.1]{AP16}. It provides a variational interpretation of the prescribed Ricci curvature equation~\eqref{eq_PRC} on homogeneous spaces.

\begin{lemma}\label{lem_var}
Given $T\in\mathcal M$, formula~(\ref{eq_PRC}) holds for some $c\in\mathbb R$ if and only if $g$ is a critical point of the restriction of the functional $S$ to $\mathcal M_T$.
\end{lemma}

We will use this lemma in the proof of our main result, Theorem~\ref{thm_PRC}.

\begin{remark}
The restriction of $S$ to $\mathcal M_T$ is bounded above for every $T\in\mathcal M$. This is a consequence of~\cite[Equation~(1.3)]{MWWZ86} and the definition of~$\mathcal M_T$; cf.~Lemma~\ref{lem_smpl_bd} below. If the homogeneous space $M$ is effective and $T$ lies in~$\mathcal M$, then the following statements are equivalent:
\begin{enumerate}
\item
The restriction of $S$ to $\mathcal M_T$ is bounded below.
\item
The universal cover of $M$ is the product of several isotropy irreducible homogeneous spaces and a Euclidean space.
\end{enumerate}
One can prove this equivalence by repeating the argument from~\cite[Proof of Theorem~(2.1)]{MWWZ86} with minor modifications. If the two statements above hold, then all the metrics in $\mathcal M$ have the same Ricci curvature; see~\cite[Lemma~3.2]{APYRsubm}. In this case, the analysis of~\eqref{eq_PRC} is easy.
\end{remark}

Given a bilinear form $R$ on $\mathfrak m$ and a nonzero subspace $\mathfrak u\subset\mathfrak m$, we write $R|_{\mathfrak u}$ for the restriction of $R$ to~$\mathfrak u$. Let $\tr_QR|_{\mathfrak u}$ be the trace of $R|_{\mathfrak u}$ with respect to~$Q|_{\mathfrak u}$. If $R'$ is a bilinear form on $\mathfrak u$, denote
\begin{align}\label{lambda_pm_def}
\lambda_-(R')&=\inf\{R'(X,X)\,|\,X\in\mathfrak u~\mbox{and}~Q(X,X)=1\},\notag \\
\lambda_+(R')&=\sup\{R'(X,X)\,|\,X\in\mathfrak u~\mbox{and}~Q(X,X)=1\}.
\end{align}
Thus, $\lambda_-(R')$ and $\lambda_+(R')$ are the smallest and the largest eigenvalue of the matrix of $R'$ in a $Q|_{\mathfrak u}$-orthonormal basis of $\mathfrak u$. We will use the notation
\begin{align*}
\omega(\mathfrak u)=\min\{\dim\mathfrak v\,|\,\mathfrak v~\mbox{is a nonzero}\,\Ad(H)\mbox{-invariant subspace of}~\mathfrak u\}.
\end{align*}
It is clear that $\omega(\mathfrak u)$ always lies between~1 and~$\dim\mathfrak u$. In fact, $\omega(\mathfrak u)$ equals $\dim\mathfrak u$ if $\Ad(H)|_{\mathfrak u}$ is irreducible.

Given $\Ad(H)$-invariant subspaces $\mathfrak u\subset\mathfrak m$, $\mathfrak v\subset\mathfrak m$ and $\mathfrak w\subset\mathfrak m$, define a tensor $\Delta(\mathfrak u,\mathfrak v,\mathfrak w)\in\mathfrak u\otimes\mathfrak v^*\otimes\mathfrak w^*$ by the formula
\begin{align}\label{def_DeltaBig}
\Delta(\mathfrak u,\mathfrak v,\mathfrak w)(X,Y)=\pi_{\mathfrak u}[X,Y],\qquad X\in\mathfrak v,~Y\in\mathfrak w.
\end{align}
Here and in what follows, $\pi_{\mathfrak u}$ stands for the $Q$-orthogonal projection onto~$\mathfrak u$. Let $\langle \mathfrak u\mathfrak v\mathfrak w\rangle$ be the squared norm of $\Delta(\mathfrak u,\mathfrak v,\mathfrak w)$ with respect to the scalar product on $\mathfrak u\otimes\mathfrak v^*\otimes\mathfrak w^*$ induced by $Q|_{\mathfrak u}$, $Q|_{\mathfrak v}$ and $Q|_{\mathfrak w}$. The fact that $Q$ comes from a bi-invariant metric on $G$ implies
\begin{align*}
\langle\mathfrak u\mathfrak v\mathfrak w\rangle=\langle\mathfrak w\mathfrak u\mathfrak v\rangle=\langle\mathfrak v\mathfrak w\mathfrak u\rangle=\langle\mathfrak v\mathfrak u\mathfrak w\rangle=\langle\mathfrak u\mathfrak w\mathfrak v\rangle=\langle\mathfrak w\mathfrak v\mathfrak u\rangle.
\end{align*}
It is easy to compute $\langle \mathfrak u\mathfrak v\mathfrak w\rangle$ in terms of the structure constants of the homogeneous space $M$; see formula~\eqref{<ijk>[ijk]} below.

\subsection{The sufficient condition}\label{sec_result}

Our main result, Theorem~\ref{thm_PRC}, requires the following hypothesis. The class of homogeneous spaces for which this hypothesis holds is very broad. We discuss examples in Sections~\ref{sec_2sum} and~\ref{sec_flag}.

\begin{hypothesis}\label{hyp_flag}
Every Lie subalgebra $\mathfrak s\subset\mathfrak g$ such that $\mathfrak h\subset\mathfrak s$ and $\mathfrak h\ne\mathfrak s$ meets the following requirements:
\begin{enumerate}
\item
The representations $\Ad(H)|_{\mathfrak u}$ and $\Ad(H)|_{\mathfrak v}$ are inequivalent for every pair of nonzero $\Ad(H)$-invariant spaces $\mathfrak u\subset\mathfrak s\ominus\mathfrak h$ and $\mathfrak v\subset\mathfrak g\ominus\mathfrak s$.

\item
The commutator $[\mathfrak r,\mathfrak s]$ is nonzero for every $\Ad(H)$-invariant 1-dimensional subspace $\mathfrak r$ of $\mathfrak g\ominus\mathfrak s$.
\end{enumerate}
\end{hypothesis}

\begin{remark}
One can show that requirement~1 of Hypothesis~\ref{hyp_flag} holds for every $\mathfrak s$ if the isotropy representation of $M$ splits into pairwise inequivalent irreducible summands; cf.~the proof of Proposition~\ref{prop_g_flag_mfs} below. However, this requirement may be satisfied (at least, for some $\mathfrak s$) even if $M$ does not possess this property. To give an example, suppose $H=SO(k-2)$ with $k\ge4$ embedded naturally into $G=SO(k)$. Then $M$ is the Stiefel manifold $V_2\mathbb R^{k}$. Let $\mathfrak s$ be the direct sum of $\mathfrak{so}_2$ and $\mathfrak h=\mathfrak{so}_{k-2}$ embedded naturally into $\mathfrak g=\mathfrak{so}_k$. Then the representation $\Ad(H)|_{\mathfrak s\ominus\mathfrak h}$ is trivial, while the representation $\Ad(H)|_{\mathfrak g\ominus\mathfrak s}$ splits into two equivalent $(k-2)$-dimensional irreducible summands; see~\cite[Section~4]{MK98}.
\end{remark}

\begin{remark}
In a sense, requirement~2 of Hypothesis~\ref{hyp_flag} is necessary for Theorem~\ref{thm_PRC} to hold. We explain this after the proof of Lemma~\ref{lem_eta_wdpos}.
\end{remark}

\begin{remark}\label{rem_1d_alg}
Suppose $\mathfrak r$ is an $\Ad(H)$-invariant 1-dimensional subspace of $\mathfrak g\ominus\mathfrak s$. If the commutator $[\mathfrak r,\mathfrak s]$ equals $\{0\}$, then the direct sum of $\mathfrak r$ and $\mathfrak s$ is a Lie subalgebra of $\mathfrak g$ isomorphic to the direct sum of $\mathbb R$ and~$\mathfrak s$. It is obvious that requirement~2 of Hypothesis~\ref{hyp_flag} holds for $\mathfrak s$ if no such subalgebra exists.
\end{remark}

\begin{remark}
In Section~\ref{sec_flag}, we will encounter cases where $\mathfrak g\ominus\mathfrak s$ does not have any $\Ad(H)$-invariant 1-dimensional subspaces. In these cases, requirement~2 of Hypothesis~\ref{hyp_flag} is automatically satisfied for~$\mathfrak s$.
\end{remark}

Suppose $\mathfrak k$ and $\mathfrak k'$ are Lie subalgebras of $\mathfrak g$ such that 
\begin{align}\label{flag}
\mathfrak g\supset\mathfrak k\supset\mathfrak k'\supset\mathfrak h.
\end{align}
In order to state our main result, we need to introduce some terminology and notation.

\begin{definition}\label{def_simple_chain}
We call~(\ref{flag}) a \emph{simple chain} if $\mathfrak k'$ is a maximal Lie subalgebra of $\mathfrak k$ and $\mathfrak h\ne \mathfrak k'$.
\end{definition}

Let us emphasise that Definition~\ref{def_simple_chain} allows the equality $\mathfrak k=\mathfrak g$ but not $\mathfrak k'=\mathfrak k$. We denote
\begin{align}\label{def_funny_subsp}
\mathfrak j=\mathfrak g\ominus\mathfrak k,\qquad \mathfrak j'=\mathfrak g\ominus\mathfrak k',\qquad \mathfrak l=\mathfrak k\ominus\mathfrak k',\qquad \mathfrak n=\mathfrak k'\ominus\mathfrak h.
\end{align}
It is obvious that
\begin{align*}
\mathfrak g=\mathfrak j\oplus\mathfrak l\oplus\mathfrak n\oplus\mathfrak h=\mathfrak j'\oplus\mathfrak n\oplus\mathfrak h,\qquad \mathfrak j'=\mathfrak j\oplus\mathfrak l,\qquad \mathfrak k=\mathfrak l\oplus\mathfrak n\oplus\mathfrak h,\qquad \mathfrak k'=\mathfrak n\oplus\mathfrak h.
\end{align*}
Here and in what follows, the symbol $\oplus$ stands for the $Q$-orthogonal sum.

Suppose~\eqref{flag} is a simple chain. In order to state our main result, we need to associate a number, denoted $\eta(\mathfrak k,\mathfrak k')$, to this simple chain. Let $B$ be the Killing form of the Lie algebra $\mathfrak g$. Define $\eta(\mathfrak k,\mathfrak k')$ by the formula
\begin{align}\label{def_etai}
\eta(\mathfrak k,\mathfrak k')=\frac{2\tr_QB|_{\mathfrak n}+2\langle\mathfrak n\mathfrak j'\mathfrak j'\rangle+\langle\mathfrak n\mathfrak n\mathfrak n\rangle}{\omega(\mathfrak n)(2\tr_QB|_{\mathfrak l}+\langle\mathfrak l\mathfrak l\mathfrak l\rangle+2\langle\mathfrak l\mathfrak j\mathfrak j\rangle)}.
\end{align}
Lemma~\ref{lem_eta_wdpos} below shows, when Hypothesis~\ref{hyp_flag} is satisfied, that the denominator in~\eqref{def_etai} can never equal~0 and that $\eta(\mathfrak k,\mathfrak k')\ge0$. We are now ready to formulate the main result of the present paper. We prove it in Subsections~\ref{sec_lemmas}--\ref{sec_final2}.

\begin{theorem}\label{thm_PRC}
Suppose Hypothesis~\ref{hyp_flag} is satisfied for the homogeneous space~$M$. Consider a tensor field $T\in\mathcal M$. If the inequality
\begin{align}\label{ineq_main_thm}
\frac{\lambda_-(T|_{\mathfrak n})}{\tr_Q T|_{\mathfrak l}}>\eta(\mathfrak k,\mathfrak k')
\end{align}
holds for every simple chain of the form~(\ref{flag}), then there exists a Riemannian metric $g\in\mathcal M_T$ such that $S(g)\ge S(h)$ for all $h\in\mathcal M_T$. The Ricci curvature of $g$ coincides with $cT$ for some $c>0$.
\end{theorem}

Subsection~\ref{sec_lemmas} contains simple and ``practical" formulas for the quantities appearing in~\eqref{ineq_main_thm}. Specifically, the eigenvalue $\lambda_-(T|_{\mathfrak n})$ and the trace $\tr_Q T|_{\mathfrak l}$ are given by~\eqref{lT_trT_smpl_flas}, while the computation of $\eta(\mathfrak k,\mathfrak k')$ on concrete homogeneous spaces is likely to involve~\eqref{eq_omega_smpl}, \eqref{trace_bd} and~\eqref{<ijk>[ijk]}. One can also find $\eta(\mathfrak k,\mathfrak k')$ with the aid of Lemma~\ref{lem_eta_wdpos}.

In Sections~\ref{sec_2sum} and~\ref{sec_flag}, we discuss several classes of examples that illustrate the use of Theorem~\ref{thm_PRC}. As part of this discussion, we compute the numbers $\eta(\mathfrak k,\mathfrak k')$ explicitly for all simple chains on certain generalised flag manifolds. In Subsection~\ref{subsec_cor}, we state two corollaries of Theorem~\ref{thm_PRC}. One of them provides an alternative to~\eqref{ineq_main_thm}, and the other deals with the case where~\eqref{ineq_main_thm} holds for all $T\in\mathcal M$.

\begin{remark}
Theorem~\ref{thm_PRC} assumes that the tensor field $T$ is positive-definite. Let us make a few comments related to this assumption. If $T$ is degenerate, then the restriction of $S$ to $\mathcal M_T$ may be unbounded above. This is possible even if $M$ satisfies Hypothesis~\ref{hyp_flag}; see~\cite[Remark~3.2]{AP16} for a class of examples. If $T$ has mixed signature, the techniques used in our proof of Theorem~\ref{thm_PRC} appear to be ineffective. Particularly, the estimates in Lemmas~\ref{lem_est_max}, \ref{lem_smpl_bd} and~\ref{lem_eps+max} seem to break down. Finally, if $T$ is negative-definite, a Riemannian metric $g\in\mathcal M_T$ with Ricci curvature $cT$ does not exist for any~$c>0$. This is a consequence of Bochner's theorem; see~\cite[Theorem~1.84]{AB87}.
\end{remark}

\begin{remark}
Given $T\in\mathcal M$, if $\mathfrak h$ is not a maximal Lie subalgebra of $\mathfrak g$, Hypothesis~\ref{hyp_flag} is satisfied, and~\eqref{ineq_main_thm} holds for every simple chain of the form~\eqref{flag}, then the restriction of $S$ to $\mathcal M_T$ cannot be proper. This observation follows from Remark~\ref{rem_proper} and Lemma~\ref{lem_ind_dim_subgr} below. In a sense, it is an analogue of the ``only if" part of~\cite[Theorem~(2.2)]{MWWZ86}, a result concerning the restriction of $S$ to the set $\mathcal M_1$ given by~\eqref{intro_def_M1}.
\end{remark}

\subsection{Some background and preparatory lemmas}\label{sec_lemmas}

The background material in this subsection is mostly standard. It is presented in greater detail in, for example,~\cite{MWWZ86,YNERVS07}. However, to the best of the authors' knowledge, Lemmas~\ref{lem_k=m}, \ref{lem_vanish_brak} and \ref{lem_eta_wdpos}, as well as Proposition~\ref{lem_eta0}, are new.

Throughout Subections~\ref{sec_lemmas}--\ref{sec_final2}, we assume Hypothesis~\ref{hyp_flag} holds. Some of our lemmas can actually be proven under milder conditions than those imposed. This is explained in Remark~\ref{rem_milder}. As above, throughout Subsections~\ref{sec_lemmas}--\ref{sec_proofs}, we suppose $\mathfrak k$ and $\mathfrak k'$ are distinct Lie subalgebras of $\mathfrak g$ satisfying the inclusions $\mathfrak h\subset\mathfrak k'\subset\mathfrak k$. However, unless stated otherwise, we do not require~\eqref{flag} to be a simple chain. The spaces $\mathfrak j$, $\mathfrak j'$, $\mathfrak l$ and $\mathfrak n$ are defined by~\eqref{def_funny_subsp}.

Consider a $Q$-orthogonal $\Ad(H)$-invariant decomposition
\begin{align}\label{m_decomp}
\mathfrak m=\mathfrak m_1\oplus\cdots\oplus\mathfrak m_s
\end{align}
such that $\Ad(H)|_{\mathfrak m_i}$ is irreducible for each $i=1,\ldots,s$. Let $d_i$ denote the dimension of $\mathfrak m_i$. Generally speaking, the space $\mathfrak m$ admits more than one decomposition of the form~\eqref{m_decomp}. However, the number $s$ and the multiset $\{d_1,\ldots,d_s\}$ must be the same for all such decompositions.

The summands $\mathfrak m_1,\ldots,\mathfrak m_s$ are determined uniquely up to order if $\Ad(H)|_{\mathfrak m_i}$ is inequivalent to $\Ad(H)|_{\mathfrak m_j}$ whenever $i\ne j$. This fact can be derived from Schur's lemma; see, e.g.,~\cite[Subsection~2.1]{APYRsubm}.

Our analysis will rely heavily on the following consequence of Hypothesis~\ref{hyp_flag}.

\begin{lemma}\label{lem_k=m}
There exists a set $J_{\mathfrak k}\subset\{1,\ldots,s\}$ satisfying the equality
\begin{align}\label{eq_bigoplus}
\mathfrak k\ominus\mathfrak h=\bigoplus_{j\in J_{\mathfrak k}}{\mathfrak m}_j.
\end{align}
Evidently, such a set is unique.
\end{lemma}

Throughout the paper, we assume 
\begin{align*}
\bigoplus_{j\in\emptyset}{\mathfrak m}_j=\{0\}.
\end{align*}

\begin{proof}[Proof of Lemma~\ref{lem_k=m}]
Fix a $Q$-orthogonal $\Ad(H)$-invariant 
decomposition
\begin{align*}
\mathfrak m&=\mathfrak m'_1\oplus\cdots\oplus\mathfrak m'_s
\end{align*}
such that $\Ad(H)|_{\mathfrak m'_j}$ is irreducible for each $j=1,\ldots,s$ and
\begin{align*}
\mathfrak k\ominus\mathfrak h=\mathfrak m'_1\oplus\cdots\oplus\mathfrak m'_p
\end{align*}
for some $p=1,\ldots,s$. One can easily verify that such a decomposition exists. Consider the map $\pi_{jk}:\mathfrak m_j\to\mathfrak m'_k$ sending a vector in $\mathfrak m_j$ to its $Q$-orthogonal projection onto $\mathfrak m_k'$. Clearly, this map is $\Ad(H)$-invariant for all $j,k=1,\ldots,s$. It is, therefore, an isomorphism or zero by Schur's lemma. Define
\begin{align*}
J_{\mathfrak k}=\{j\in[1,s]\cap\mathbb N\,|\,\pi_{jk}~\mbox{is an isomorphism for some}~k\in[1,p]\cap\mathbb N\}.
\end{align*}
We claim that~\eqref{eq_bigoplus} holds. To prove this, we first fix $k\le p$ and show that 
\begin{align}\label{inclopl}
\mathfrak m_k'\subset\bigoplus_{j\in J_{\mathfrak k}}\mathfrak m_j.
\end{align}

Consider the map $\pi_{kl}':\mathfrak m_k'\to\mathfrak m_l$ sending a vector in $\mathfrak m'_k$ to its $Q$-orthogonal projection onto $\mathfrak m_l$. Choose $X\in\mathfrak m_k'$. The equality
\begin{align*}
X=\pi'_{k1}X+\cdots+\pi'_{ks}X
\end{align*}
holds true. To prove formula~\eqref{inclopl}, it suffices to show that $l\in J_{\mathfrak k}$ whenever $\pi'_{kl}X\ne0$. Clearly, $\mathfrak m_k'$ is not orthogonal to $\mathfrak m_l$ if $\pi'_{kl}X\ne0$. Therefore, $\pi_{lk}\ne0$ if this inequality holds. Schur's lemma then implies that $\pi_{lk}$ must be an isomorphism. Therefore, $l$ lies in $J_{\mathfrak k}$, formula~\eqref{inclopl} holds, and $\mathfrak k\ominus\mathfrak h$ is a subset of~$\bigoplus_{j\in J_{\mathfrak k}}\mathfrak m_j$.

We now fix $k>p$ and $l\in J_{\mathfrak k}$. Our next step is to prove that $Q(\mathfrak m'_k,\mathfrak m_l)=\{0\}$. This equality implies that the $Q$-orthogonal complement of $\bigoplus_{j\in J_{\mathfrak k}}\mathfrak m_j$ contains the $Q$-orthogonal complement of $\mathfrak k$. This fact, in its turn, shows that $\bigoplus_{j\in J_{\mathfrak k}}\mathfrak m_j$ is a subset of $\mathfrak k\ominus\mathfrak h$. Consequently, formula~\eqref{eq_bigoplus} holds.

Assume $Q(\mathfrak m'_k,\mathfrak m_l)\ne\{0\}$. By Schur's lemma, the map $\pi'_{kl}$ is then an isomorphism. Since $l$ lies in~$J_{\mathfrak k}$, there exists $q\le p$ such that $\pi_{lq}$ is an isomorphism as well. Evidently, $k\ne q$. Consider the map $\pi_{lq}\pi'_{kl}:\mathfrak m'_k\to\mathfrak m_q$. It is an $\Ad(H)$-invariant isomorphism. However, the existence of such an isomorphism contradicts requirement~1 of Hypothesis~\ref{hyp_flag}.
\end{proof}

\begin{corollary}\label{cor_finite_subalg}
The Lie algebra $\mathfrak g$ has at most $2^s$ distinct Lie subalgebras containing $\mathfrak h$.
\end{corollary}

Define $J_{\mathfrak h}=\emptyset$. Observe that $J_{\mathfrak g}=\{1,\ldots,s\}$. It will be convenient for us to set
\begin{align}\label{def_diff_Js}
J_{\mathfrak j}=J_{\mathfrak g}\setminus J_{\mathfrak k},\qquad J_{\mathfrak j'}=J_{\mathfrak g}\setminus J_{\mathfrak k'},\qquad J_{\mathfrak l}=J_{\mathfrak k}\setminus J_{\mathfrak k'}.
\end{align}
Evidently,
\begin{align}\label{eq_lsum_jsum}
\mathfrak j=\bigoplus_{j\in J_{\mathfrak j}}\mathfrak m_j,\qquad \mathfrak j'=\bigoplus_{j\in J_{\mathfrak j'}}\mathfrak m_j,\qquad\mathfrak l=\bigoplus_{j\in J_{\mathfrak l}}\mathfrak m_j,\qquad 
\mathfrak n=\bigoplus_{j\in J_{\mathfrak k'}}\mathfrak m_j,
\end{align}
which implies
\begin{align}\label{eq_omega_smpl}
\omega(\mathfrak n)=\min_{j\in J_{\mathfrak k'}}d_j.
\end{align}
Given $T\in\mathcal M$, it is always possible to choose the decomposition~\eqref{m_decomp} so that
\begin{align*}
T=\sum_{i=1}^sz_i\pi_{\mathfrak m_i}^*Q,\qquad z_i>0;
\end{align*}
see~\cite[page~180]{MWWZ86}. If this formula holds, then
\begin{align}\label{lT_trT_smpl_flas}
\lambda_-(T|_{\mathfrak n})=\min_{i\in J_{\mathfrak k'}}z_i,\qquad 
\tr_QT|_{\mathfrak l}=\sum_{i\in J_{\mathfrak l}}d_iz_i.
\end{align}

Recall that $B$ denotes the Killing form of~$\mathfrak g$. For every $i=1,\ldots,s$, because $\Ad(H)|_{\mathfrak m_i}$ is irreducible, there exists $b_i\ge0$ such that
\begin{align}\label{b_def}
B|_{\mathfrak m_i} = -b_iQ|_{\mathfrak m_i}.
\end{align}
It is clear that
\begin{align}\label{trace_bd}
\tr_QB|_{\mathfrak n}=-\sum_{j\in J_{\mathfrak k'}}d_jb_j,\qquad \tr_QB|_{\mathfrak l}=-\sum_{j\in J_{\mathfrak l}}d_jb_j.
\end{align}
Given $i,j,k\in\{1,\ldots,s\}$, define
\begin{align*}
[ijk]=\langle\mathfrak m_i\mathfrak m_j\mathfrak m_k\rangle.
\end{align*}
Note that $[ijk]$ is symmetric in all three indices. The numbers $([ijk])_{i,j,k=1}^s$ are often called the \emph{structure constants} of the homogeneous space $M$. If
\begin{align}\label{eq_uvw_sum}
\mathfrak u=\bigoplus_{i\in \mathcal J_{\mathfrak u}}\mathfrak m_i, \qquad \mathfrak v=\bigoplus_{i\in \mathcal J_{\mathfrak v}}\mathfrak m_i,\qquad \mathfrak w=\bigoplus_{i\in \mathcal J_{\mathfrak w}}\mathfrak m_i,
\end{align}
where $\mathcal J_{\mathfrak u}$, $\mathcal J_{\mathfrak v}$ and $\mathcal J_{\mathfrak w}$ are subsets of $\{1,\ldots,s\}$, then
\begin{align}\label{<ijk>[ijk]}
\langle\mathfrak u\mathfrak v\mathfrak w\rangle=\sum_{i\in\mathcal J_{\mathfrak u}}\sum_{j\in\mathcal J_{\mathfrak v}}\sum_{k\in\mathcal J_{\mathfrak w}}[ijk].
\end{align}
(We interpret the sum over the empty set as~0.)

\begin{lemma}\label{lem_vanish_brak}
If $i\in J_{\mathfrak l}$ and $j,k\in J_{\mathfrak k'}$, then $[ijk]=0$.
\end{lemma}

\begin{proof}
The inclusion $j,k\in J_{\mathfrak k'}$ implies that $\mathfrak m_j$ and $\mathfrak m_k$ are subspaces of the Lie algebra $\mathfrak k'$. Therefore, the map 
\begin{align*}
\mathfrak m_j\times\mathfrak m_k\ni(X,Y)\mapsto[X,Y]
\end{align*}
takes values in $\mathfrak k'$. Since $i\in J_{\mathfrak l}$, the $Q$-orthogonal projection of $\mathfrak k'$ onto $\mathfrak m_i$ equals $\{0\}$. This means the tensor $\Delta(\mathfrak m_i,\mathfrak m_j,\mathfrak m_k)$ given by~\eqref{def_DeltaBig} is the zero tensor. Thus, the assertion of the lemma holds.
\end{proof}

Fix a $Q$-orthonormal basis $(w_j)_{j=1}^{\dim\mathfrak h}$ of the Lie algebra $\mathfrak h$. Given $i=1,\ldots,s$, consider the Casimir operator $C_{\mathfrak m_i,Q|_{\mathfrak h}}:\mathfrak m_i\to\mathfrak m_i$ defined by the formula
\begin{align*}
C_{\mathfrak m_i,Q|_{\mathfrak h}}(X)=-\bigg(\sum_{j=1}^{\dim\mathfrak h}\ad w_j\circ\ad w_j\bigg)(X), \qquad X\in\mathfrak m_i.
\end{align*}
The irreducibility of $\Ad(H)|_{\mathfrak m_i}$ implies the existence of $\zeta_i\ge0$ such that
\begin{align}\label{zeta_def}
C_{\mathfrak m_i,Q|_{\mathfrak h}}(X)=\zeta_iX, \qquad X\in\mathfrak m_i.
\end{align}
Note that $\zeta_i=0$ if and only if $\Ad(H)|_{\mathfrak m_i}$ is trivial. According to~\cite[Lemma~(1.5)]{MWWZ86}, the arrays 
$(b_i)_{i=1}^s$, $([ijk])_{i,j,k=1}^s$ and $(\zeta_i)_{i=1}^s$
are related to each other by the equality 
\begin{align}\label{Casimir}
d_ib_i=2d_i\zeta_i+\sum_{j,k=1}^s[ijk].
\end{align}
The following result shows that the numbers $\eta(\mathfrak k,\mathfrak k')$ introduced in Subsection~\ref{sec_result} are well-defined and non-negative. 

\begin{lemma}\label{lem_eta_wdpos}
One has
\begin{align*}
-2\tr_QB|_{\mathfrak l}-\langle \mathfrak l\mathfrak l\mathfrak l\rangle-2\langle \mathfrak l\mathfrak j\mathfrak j\rangle&=\sum_{j\in J_{\mathfrak l}}\bigg(4d_j\zeta_j+\sum_{k,l\in J_{\mathfrak l}}[jkl]+4\sum_{k\in J_{\mathfrak k'}}\sum_{l\in J_{\mathfrak l}}[jkl]\bigg)>0, \\
-2\tr_QB|_{\mathfrak n}-2\langle\mathfrak n\mathfrak j'\mathfrak j'\rangle-\langle\mathfrak n\mathfrak n\mathfrak n\rangle
&=\sum_{j\in J_{\mathfrak k'}}\bigg(4d_j\zeta_j+\sum_{k,l\in J_{\mathfrak k'}}[jkl]\bigg)\ge0.
\end{align*}
\end{lemma}

\begin{proof}
Equalities~\eqref{trace_bd}, \eqref{<ijk>[ijk]} and~\eqref{Casimir}, together with Lemma~\ref{lem_vanish_brak}, yield
\begin{align}\label{eq_etai_nonz}
-2\tr_QB|_{\mathfrak l}-\langle \mathfrak l\mathfrak l\mathfrak l\rangle-2\langle \mathfrak l\mathfrak j\mathfrak j\rangle&= 2\sum_{j\in J_{\mathfrak l}}d_jb_j-\sum_{j,k,l\in J_{\mathfrak l}}[jkl]-2\sum_{j\in J_{\mathfrak l}}\sum_{k,l\in J_{\mathfrak j}}[jkl]\notag
\\
&=\sum_{j\in J_{\mathfrak l}}\bigg(4d_j\zeta_j+\sum_{k,l\in J_{\mathfrak l}}[jkl]+4\sum_{k\in J_{\mathfrak k'}}\sum_{l\in J_{\mathfrak l}}[jkl]\bigg).
\end{align}
The expression in the last line must be non-negative because the numbers $d_j$, $\zeta_j$ and $[jkl]$ are non-negative by definition. If it is~0, then $\zeta_j=0$ for every $j\in J_{\mathfrak l}$. Consequently, the representation $\Ad(H)|_{\mathfrak m_j}$ is trivial for every such~$j$. Since $\Ad(H)|_{\mathfrak m_j}$ is also irreducible, this means $d_j=1$.
Moreover, in view of Lemma~\ref{lem_vanish_brak}, if the expression in the last line of~\eqref{eq_etai_nonz} is~0, then
\begin{align*}
[jkl]=0,\qquad j\in J_{\mathfrak l},~k,l\in J_{\mathfrak k}.
\end{align*}
This implies
\begin{align*}
[\mathfrak m_j,\mathfrak k']=\{0\},\qquad j\in J_{\mathfrak l}.
\end{align*}
However, the commutation $[\mathfrak m_j,\mathfrak k']$ must be non-trivial by requirement~2 of Hypothesis~\ref{hyp_flag}. Thus, the expression in the last line of~\eqref{eq_etai_nonz} cannot be~0, and the first formula in the statement of the lemma holds.

Next, we use~\eqref{trace_bd}, \eqref{<ijk>[ijk]}, \eqref{Casimir} and Lemma~\ref{lem_vanish_brak} again to compute
\begin{align*}
-2\tr_QB|_{\mathfrak n}&-2\langle\mathfrak n\mathfrak j'\mathfrak j'\rangle-\langle\mathfrak n\mathfrak n\mathfrak n\rangle
\\
&=2\sum_{j\in J_{\mathfrak k'}}d_jb_j-2\sum_{j\in J_{\mathfrak k'}}\sum_{k,l\in J_{\mathfrak j'}}[jkl]-\sum_{j,k,l\in J_{\mathfrak k'}}[jkl]=\sum_{j\in J_{\mathfrak k'}}\bigg(4d_j\zeta_j+\sum_{k,l\in J_{\mathfrak k'}}[jkl]\bigg)\ge0.
\end{align*}
\end{proof}

\begin{remark}\label{rem_denom_0}
If $\mathfrak g$ had a Lie subalgebra $\mathfrak s$ containing $\mathfrak h$ as a proper subset and satisfying the first requirement of Hypothesis~\ref{hyp_flag} but not the second, then the formulation of Theorem~\ref{thm_PRC} would become meaningless. Indeed, in this case, it would be possible to find an $\Ad(H)$-invariant 1-dimensional subspace $\mathfrak r$ of $\mathfrak g\ominus\mathfrak s$ such that $[\mathfrak r,\mathfrak s]=\{0\}$. By Remark~\ref{rem_1d_alg},
\begin{align*}
\mathfrak g\supset\mathfrak r\oplus\mathfrak s\supset\mathfrak s\supset\mathfrak h
\end{align*}
would be a simple chain. However, employing~\eqref{eq_etai_nonz}, we would be able to demonstrate that $\eta(\mathfrak r\oplus\mathfrak s,\mathfrak s)$ is not well-defined.
\end{remark}

The following result provides insight into the nature of the numbers~$\eta(\mathfrak k,\mathfrak k')$. It will help us establish a corollary of Theorem~\ref{thm_PRC} in Subsection~\ref{subsec_cor}.

\begin{proposition}\label{lem_eta0}
Assume~(\ref{flag}) is a simple chain. The number $\eta(\mathfrak k,\mathfrak k')$ is~0 if and only if the Lie algebra $\mathfrak k'$ is isomorphic to the direct sum of $\mathbb R$ and~$\mathfrak h$.
\end{proposition}

\begin{proof}
Assume $\eta(\mathfrak k,\mathfrak k')=0$. This means the numerator in~\eqref{def_etai} must be~0. Therefore, in view of Lemma~\ref{lem_eta_wdpos}, 
\begin{align*}
\sum_{j\in J_{\mathfrak k'}}\bigg(4d_j\zeta_j+\sum_{k,l\in J_{\mathfrak k'}}[jkl]\bigg)=0.
\end{align*}
Since the numbers $d_j$, $\zeta_j$ and $[jkl]$ are all non-negative, $\zeta_j=0$ for all $j\in J_{\mathfrak k'}$. As a consequence, the representation $\Ad(H)|_{\mathfrak m_j}$ is trivial for such~$j$. We will use this fact to prove that $\mathfrak k'$ is isomorphic to the direct sum of $\mathbb R$ and~$\mathfrak h$.

Fix $i\in J_{\mathfrak k'}$. The irreducibility of $\Ad(H)|_{\mathfrak m_i}$ implies that the dimension $d_i$ equals~1. Consequently, 
\begin{align*}
\mathfrak k''=\mathfrak m_i\oplus\mathfrak h
\end{align*}
is a Lie subalgebra of $\mathfrak k'$. Our next step is to show that $\mathfrak k''$ is, in fact, equal to~$\mathfrak k'$.

Choose $k\in J_{\mathfrak k'}$. The dimension of $\mathfrak m_k$ is~1. Because the representations $\Ad(H)|_{\mathfrak m_i}$ and $\Ad(H)|_{\mathfrak m_k}$ are both trivial, they are equivalent. Clearly, $\mathfrak m_i$ coincides with $\mathfrak k''\ominus\mathfrak h$. If $k\ne i$, then $\mathfrak m_k$ must lie in $\mathfrak g\ominus\mathfrak k''$. However, this means $\mathfrak k''$ does not meet requirement~1 of Hypothesis~\ref{hyp_flag}. Thus, $i$ is the only element in~$J_{\mathfrak k'}$. We conclude that $\mathfrak k''$ equals $\mathfrak k'$. It is clear that $\mathfrak k''$ is isomorphic to the direct sum of $\mathbb R$ and~$\mathfrak h$. This proves the ``only if" portion of the lemma. Next, we turn to the converse statement.

Assume $\mathfrak k'$ is isomorphic to the direct sum of $\mathbb R$ and~$\mathfrak h$. Let us show that $\eta(\mathfrak k,\mathfrak k')=0$. According to~\eqref{def_etai} and Lemma~\ref{lem_eta_wdpos}, 
\begin{align}\label{eq_eta_zetas_braks}
\eta(\mathfrak k,\mathfrak k')=\frac{\sum_{j\in J_{\mathfrak k'}}\big(4d_j\zeta_j+\sum_{k,l\in J_{\mathfrak k'}}[jkl]\big)}{\omega(\mathfrak n)\sum_{j\in J_{\mathfrak l}}\big(4d_j\zeta_j+\sum_{k,l\in J_{\mathfrak l}}[jkl]+4\sum_{k\in J_{\mathfrak k'}}\sum_{l\in J_{\mathfrak l}}[jkl]\big)}.
\end{align}
The proof will be complete if we demonstrate that the numerator is~0.

Lemma~\ref{lem_k=m} and the existence of an isomorphism between $\mathfrak k'$ and the direct sum of $\mathbb R$ and~$\mathfrak h$ imply that
\begin{align*}
\mathfrak k'=\mathfrak m_i\oplus\mathfrak h
\end{align*}
for some $i=1,\ldots,s$. Moreover, the dimension of $\mathfrak m_i$ is~1. Consequently, $J_{\mathfrak k'}$ is the set~$\{i\}$, and
\begin{align*}
\sum_{j,k,l\in J_{\mathfrak k'}}[jkl]=[iii]=0.
\end{align*}
This formula implies that the numerator on the right-hand side of~\eqref{eq_eta_zetas_braks} equals
\begin{align*}
4\sum_{j\in J_{\mathfrak k'}}d_j\zeta_j=4d_i\zeta_i.
\end{align*}
The proof will be complete if we demonstrate that $\zeta_i=0$. It suffices to show that the representation $\Ad(H)|_{\mathfrak m_i}$ is trivial.

Choose a nonzero $X\in\mathfrak m_i$ and some $Y\in\mathfrak h$. Since $\mathfrak m_i$ is $\Ad(H)$-invariant and 1-dimensional,
the commutator $[X,Y]$ equals $\tau X$ for some $\tau\in\mathbb R$. The fact that $Q$ is induced by a bi-invariant metric on $G$ implies
\begin{align*}
\tau=\frac{Q([X,Y],X)}{Q(X,X)}=-\frac{Q([X,X],Y)}{Q(X,X)}=0.
\end{align*}
Thus, $[X,Y]$ vanishes for $X\in\mathfrak m_i$ and $Y\in\mathfrak h$, which means $\Ad(H)|_{\mathfrak m_i}$ is trivial.
\end{proof}

\subsection{The scalar curvature and related functionals}\label{sec_proofs}

The proof of Theorem~\ref{thm_PRC} relies on the analysis of two functionals related to the scalar curvature of metrics in~$\mathcal M$. Let us introduce the first of these functionals. Suppose $g$ is an $\Ad(H)$-invariant scalar product on an $\Ad(H)$-invariant subspace $\mathfrak u\subset\mathfrak m$. Define
\begin{align}\label{scal_def}
S(g)=-\frac12\tr_gB|_{\mathfrak u}-\frac14
|\Delta(\mathfrak u,\mathfrak u,\mathfrak u)|_g^2.
\end{align}
In this formula, $\Delta(\mathfrak u,\mathfrak u,\mathfrak u)$ is given by~\eqref{def_DeltaBig}, and $|\cdot|_g$ is the norm on $\mathfrak u\otimes\mathfrak u^*\otimes\mathfrak u^*$ induced by~$g$. If $\mathfrak u=\mathfrak m$, then we identify $g$ with a Riemannian metric in~$\mathcal M$. The quantity on the right-hand side of~\eqref{scal_def} is then equal to the scalar curvature of this metric; see, e.g.,~\cite[Corollary~7.39]{AB87}. Thus, the notation~\eqref{scal_def} is consistent with the notation introduced in the beginning of Subsection~\ref{subsec_nonstand_nota}. The following result provides a handy formula for $S(g)$; cf.~\cite[\S1]{MWWZ86}, \cite[Section~1]{JSPYS97} and~\cite[Section~3]{APYRsubm}.

\begin{lemma}\label{lem_sc_comp}
Let $\mathfrak u$ satisfy the first equality in~(\ref{eq_uvw_sum}) for some $\mathcal J_{\mathfrak u}\subset\{1,\ldots,s\}$. Suppose the scalar product $g$ and the decomposition~(\ref{m_decomp}) are such that 
\begin{align*}
g=\sum_{i\in\mathcal J_{\mathfrak u}}x_i\pi_{\mathfrak m_i}^*Q,\qquad x_i>0.
\end{align*}
Then
\begin{align}\label{eq_sc_comp}
\tr_gB|_{\mathfrak u}&=-\sum_{i\in \mathcal J_{\mathfrak u}}\frac{d_ib_i}{x_i}, \qquad
|\Delta(\mathfrak u,\mathfrak u,\mathfrak u)|_g^2=\sum_{i,j,k\in\mathcal J_{\mathfrak u}}[ijk]\frac{x_k}{x_ix_j}, \notag
\\
S(g)&=\frac12\sum_{i\in\mathcal J_{\mathfrak u}}\frac{d_ib_i}{x_i}-\frac14\sum_{i,j,k\in\mathcal J_{\mathfrak u}}[ijk]\frac{x_k}{x_ix_j}.
\end{align}
\end{lemma}

\begin{proof}
Let $(e_i)_{i=1}^n$ be a $Q$-orthonormal basis of $\mathfrak m$ adapted to the decomposition~\eqref{m_decomp}. For every $i=1,\ldots,n$, define $\tilde e_i=\frac1{\sqrt{x_{\iota(i)}}}e_i$, where $\iota(i)$ is the number between 1 and $s$ such that $e_i$ lies in $\mathfrak m_{\iota(i)}$. Then $(\tilde e_i)_{i=1}^n$ is a $g$-orthonormal basis of $\mathfrak m$. We compute
\begin{align*}
\tr_gB|_{\mathfrak u}&=\sum_{i\in\Gamma(\mathfrak u)}B(\tilde e_i,\tilde e_i)=\sum_{i\in\Gamma(\mathfrak u)}\frac1{x_{\iota(i)}}B(e_i,e_i)=-\sum_{i\in\mathcal J_{\mathfrak u}}\frac{d_ib_i}{x_i},
\\
|\Delta(\mathfrak u,\mathfrak u,\mathfrak u)|_g^2&=\sum_{i,j\in\Gamma(\mathfrak u)}g(\Delta(\mathfrak u,\mathfrak u,\mathfrak u)(\tilde e_i,\tilde e_j),\Delta(\mathfrak u,\mathfrak u,\mathfrak u)(\tilde e_i,\tilde e_j)) 
\\ 
&=\sum_{i,j\in\Gamma(\mathfrak u)}\sum_{k\in\mathcal J_{\mathfrak u}}\frac{x_k}{x_{\iota(i)}x_{\iota(j)}}Q(\Delta(\mathfrak m_k,\mathfrak u,\mathfrak u)(e_i,e_j),\Delta(\mathfrak m_k,\mathfrak u,\mathfrak u)(e_i,e_j))
\\
&=\sum_{i,j,k\in\mathcal J_{\mathfrak u}}[ijk]\frac{x_k}{x_ix_j}.
\end{align*}
In the first three lines,
\begin{align*}
\Gamma(\mathfrak u)=\{i\in[1,n]\cap\mathbb N\,|\,e_i\in\mathfrak u\}=\{i\in[1,n]\cap\mathbb N\,|\,\iota(i)\in\mathcal J_{\mathfrak u}\}.
\end{align*}
The last formula in~\eqref{eq_sc_comp} follows from the definition of~$S$.
\end{proof}

Let us introduce one more functional related to the scalar curvature of metrics in~$\mathcal M$. As in Subsection~\ref{sec_lemmas}, we consider distinct Lie subalgebras $\mathfrak k$ and $\mathfrak k'$ of $\mathfrak g$ such that $\mathfrak h\subset\mathfrak k'\subset\mathfrak k$. The spaces $\mathfrak j$, $\mathfrak j'$, $\mathfrak l$ and $\mathfrak n$ are given by~\eqref{def_funny_subsp}. The sets $J_\mathfrak k$, $J_{\mathfrak k'}$, $J_{\mathfrak j}$, $J_{\mathfrak j'}$ and $J_{\mathfrak l}$ appearing below are introduced in Lemma~\ref{lem_k=m} and after Corollary~\ref{cor_finite_subalg}.

Denote by $\mathcal M(\mathfrak k)$ the space of $\Ad(H)$-invariant scalar products on $\mathfrak k\ominus\mathfrak h$. There is a natural identification between $\mathcal M(\mathfrak g)$ and $\mathcal M$. In what follows, we assume $\mathcal M(\mathfrak k)$ is equipped with the topology inherited from the second tensor power of~$(\mathfrak k\ominus\mathfrak h)^*$. If $g$ lies in $\mathcal M(\mathfrak k)$, set
\begin{align*}
\hat S(g)=S(g)-\frac12|\Delta(\mathfrak j,\mathfrak k\ominus\mathfrak h,\mathfrak j)|_{QgQ}^2.
\end{align*}
The notation $|\cdot|_{QgQ}$ stands for the norm on $\mathfrak j\otimes(\mathfrak k\ominus\mathfrak h)^*\otimes\mathfrak j^*$ induced by $Q|_{\mathfrak j}$ and $g|_{\mathfrak k\ominus\mathfrak h}$. One can easily verify that $\hat S$ is a continuous map from $\mathcal M(\mathfrak k)$ to~$\mathbb R$. If $g$ lies in $\mathcal M(\mathfrak g)$, then $\hat S(g)$ equals $S(g)$.

\begin{lemma}\label{lem_hat_comp}
Suppose the scalar product $g\in\mathcal M(\mathfrak k)$ and the decomposition~(\ref{m_decomp}) are such that 
\begin{align}\label{g_diag_decm}
g=\sum_{i\in J_{\mathfrak k}}x_i\pi_{\mathfrak m_i}^*Q, \qquad x_i>0.
\end{align}
Then
\begin{align}\label{eq_hat_comp}
\hat S(g)&=\frac12\sum_{i\in J_{\mathfrak k}}\frac{d_ib_i}{x_i}-\frac12\sum_{i\in J_{\mathfrak k}}\sum_{j,k\in J_{\mathfrak j}}\frac{[ijk]}{x_i}-\frac14\sum_{i,j,k\in J_{\mathfrak k}}[ijk]\frac{x_k}{x_ix_j}.
\end{align}
\end{lemma}

\begin{proof}
As in the proof of Lemma~\ref{lem_sc_comp}, we choose a $Q$-orthonormal basis $(e_i)_{i=1}^n$ of $\mathfrak m$ adapted to the decomposition~\eqref{m_decomp}. For every $i=1,\ldots,n$, the vector $\tilde e_i$ is defined as $\frac1{\sqrt{x_{\iota(i)}}}e_i$, where $\iota(i)$ is such that $e_i\in\mathfrak m_{\iota(i)}$. To establish~\eqref{eq_hat_comp}, it suffices to take note of~\eqref{eq_sc_comp} and observe that
\begin{align*}
|\Delta(\mathfrak j,\mathfrak k\ominus\mathfrak h,\mathfrak j)|_{QgQ}^2
&=\sum_{i\in\Gamma(\mathfrak k)}\sum_{j\in\Gamma(\mathfrak j)}Q(\Delta(\mathfrak j,\mathfrak k\ominus\mathfrak h,\mathfrak j)(\tilde e_i,e_j),\Delta(\mathfrak j,\mathfrak k\ominus\mathfrak h,\mathfrak j)(\tilde e_i,e_j))
\\
&=\sum_{i\in\Gamma(\mathfrak k)}\sum_{j\in\Gamma(\mathfrak j)}\frac1{x_{\iota(i)}}Q(\Delta(\mathfrak j,\mathfrak k\ominus\mathfrak h,\mathfrak j)(e_i,e_j),\Delta(\mathfrak j,\mathfrak k\ominus\mathfrak h,\mathfrak j)(e_i,e_j))
\\
&=\sum_{i\in J_{\mathfrak k}}\sum_{j,k\in J_{\mathfrak j}}\frac{[ijk]}{x_i}.
\end{align*}
In the first two lines,
\begin{align*}
\Gamma(\mathfrak k)&=\{i\in[1,n]\cap\mathbb N\,|\,e_i\in\mathfrak k\ominus\mathfrak h\}=\{i\in[1,n]\cap\mathbb N\,|\,\iota(i)\in J_{\mathfrak k}\}.
\\
\Gamma(\mathfrak j)&=\{i\in[1,n]\cap\mathbb N\,|\,e_i\in\mathfrak j\}=\{i\in[1,n]\cap\mathbb N\,|\,\iota(i)\in J_{\mathfrak j}\}.
\end{align*}
\end{proof}

The following estimate for $S$ was essentially proven in~\cite{AP16}. Recall that the notation $\lambda_-(R')$ and $\lambda_+(R')$, where $R'$ is a bilinear form on a nonzero subspace of~$\mathfrak m$, was introduced by~\eqref{lambda_pm_def}.

\begin{lemma}\label{lem_est_max}
Suppose $\mathfrak h$ is a maximal Lie subalgebra of $\mathfrak k$. Given $g\in\mathcal M(\mathfrak k)$ and $\tau_1,\tau_2>0$, assume that
\begin{align*}
\lambda_-(g)\le\tau_1,\qquad \lambda_+(g)\ge\tau_2.
\end{align*}
Then
\begin{align*}
S(g)\le A-D\lambda_+(g)^b,
\end{align*}
where $A>0$, $D>0$ and $b>0$ are constants depending only on $G$, $H$, $\mathfrak k$, $Q$, $\tau_1$ and~$\tau_2$.
\end{lemma}

\begin{proof}
Without loss of generality, let the decomposition~\eqref{m_decomp} satisfy formula~\eqref{g_diag_decm};
cf.~\cite[page~180]{MWWZ86}. The quantity $S(g)$ is then given by Lemma~\ref{lem_sc_comp}. It is easy to see that 
\begin{align}\label{eq_lambdas_x}
\lambda_-(g)=\min_{j\in J_{\mathfrak k}}x_j,\qquad \lambda_+(g)=\max_{j\in J_{\mathfrak k}}x_j.
\end{align}
The estimate
\begin{align}\label{est_old_max}
S(g)\le\frac{\tilde A}{\min_{j\in J_{\mathfrak k}}x_j}-\frac D{(\min_{j\in J_{\mathfrak k}}x_j)^a}-D\big(\max_{j\in J_{\mathfrak k}}x_j\big)^b
\end{align}
holds with the constants $\tilde A>0$, $D>0$, $a>1$ and $b>0$ depending only on $G$, $H$, $\mathfrak k$, $Q$, $\tau_1$ and~$\tau_2$. Indeed, to obtain~\eqref{est_old_max}, it suffices to repeat the proof of~\cite[Lemma~2.4]{AP16} with only elementary modifications to the argument. The function
\begin{align*}
y\mapsto\frac{\tilde A}y-\frac D{y^a}
\end{align*}
is bounded above on $(0,\infty)$. In light of~\eqref{eq_lambdas_x} and~\eqref{est_old_max}, this fact implies
\begin{align*}
S(g)\le A-D\big(\max_{j\in J_{\mathfrak k}}x_j\big)^b=A-D\lambda_+(g)^b
\end{align*}
for some $A>0$ depending only on $G$, $H$, $\mathfrak k$, $Q$, $\tau_1$ and~$\tau_2$.
\end{proof}

We will require the following identity and estimate for $S$ and~$\hat S$.

\begin{lemma}\label{lem_elem_id_est}
Suppose the scalar product $g\in\mathcal M(\mathfrak k)$ and the decomposition~(\ref{m_decomp}) are such that~(\ref{g_diag_decm}) holds. Then
\begin{align}\label{elem_eq}
\hat S(g)&=S(g|_{\mathfrak n})+S(g|_{\mathfrak l})-\frac12\sum_{i\in J_{\mathfrak k}}\sum_{j,k\in J_{\mathfrak j}}\frac{[ijk]}{x_i}
-\frac14\sum_{i,j\in J_{\mathfrak l}}\sum_{k\in J_{\mathfrak k'}}[ijk]\Big(\frac{x_k}{x_ix_j}+2\frac{x_i}{x_jx_k}\Big), \notag
\\
\hat S(g)&\le 
\hat S(g|_{\mathfrak n})+S(g|_{\mathfrak l}).
\end{align}
\end{lemma}

\begin{proof}
By direct computation, Lemmas~\ref{lem_sc_comp} and~\ref{lem_hat_comp} imply
\begin{align*}
\hat S(g)=S(g|_{\mathfrak n})&+S(g|_{\mathfrak l})-\frac12\sum_{i\in J_{\mathfrak k}}\sum_{j,k\in J_{\mathfrak j}}\frac{[ijk]}{x_i}
\\ &-\frac14\sum_{i,j\in J_{\mathfrak l}}\sum_{k\in J_{\mathfrak k'}}[ijk]\Big(\frac{x_k}{x_ix_j}+2\frac{x_i}{x_jx_k}\Big)-\frac14\sum_{i\in J_{\mathfrak l}}\sum_{j,k\in J_{\mathfrak k'}}[ijk]\Big(2\frac{x_k}{x_ix_j}+\frac{x_i}{x_jx_k}\Big).
\end{align*}
The last of the five terms on the right-hand side vanishes. Indeed, Lemma~\ref{lem_vanish_brak} shows that the coefficients $[ijk]$ in this term are all~0. Thus, the identity in the first line of~\eqref{elem_eq} must hold. To prove the estimate, observe that
\begin{align*}
\sum_{i,j\in J_{\mathfrak l}}[ijk]\frac{x_i}{x_jx_k}
=
\frac12\sum_{i,j\in J_{\mathfrak l}}\frac{[ijk]}{x_k}\Big(\frac{x_i}{x_j}+\frac{x_j}{x_i}\Big)\ge
\sum_{i,j\in J_{\mathfrak l}}\frac{[ijk]}{x_k},\qquad k\in J_{\mathfrak k'}.
\end{align*}
Consequently,
\begin{align*}
\hat S(g)
&=S(g|_{\mathfrak n})+S(g|_{\mathfrak l})
-\frac12\sum_{i\in J_{\mathfrak k}}\sum_{j,k\in J_{\mathfrak j}}\frac{[ijk]}{x_i}-\frac14\sum_{i,j\in J_{\mathfrak l}}\sum_{k\in J_{\mathfrak k'}}[ijk]\frac{x_k}{x_ix_j}-\frac12\sum_{i,j\in J_{\mathfrak l}}\sum_{k\in J_{\mathfrak k'}}[ijk]\frac{x_i}{x_jx_k}
\\
&\le S(g|_{\mathfrak n})+S(g|_{\mathfrak l})
-\frac12\sum_{i\in J_{\mathfrak k}}\sum_{j,k\in J_{\mathfrak j}}\frac{[ijk]}{x_i}-\frac12\sum_{i,j\in J_{\mathfrak l}}\sum_{k\in J_{\mathfrak k'}}\frac{[ijk]}{x_k}
\\
&= S(g|_{\mathfrak n})+S(g|_{\mathfrak l})
-\frac12\sum_{i\in J_{\mathfrak k'}}\sum_{j,k\in J_{\mathfrak j'}}\frac{[ijk]}{x_i}-\frac12\sum_{i\in J_{\mathfrak l}}\sum_{j,k\in J_{\mathfrak j}}\frac{[ijk]}{x_i}
\\
&=\hat S(g|_{\mathfrak n})+S(g|_{\mathfrak l})-\frac12\sum_{i\in J_{\mathfrak l}}\sum_{j,k\in J_{\mathfrak j}}\frac{[ijk]}{x_i}
\le\hat S(g|_{\mathfrak n})+S(g|_{\mathfrak l}).
\end{align*}
\end{proof}

Fix $T\in\mathcal M$. Given a scalar product $g\in\mathcal M(\mathfrak k)$ and a subspace $\mathfrak u$ of $\mathfrak k\ominus\mathfrak h$, the notation $g|_{\mathfrak u}$ stands for the restriction of $g$ to $\mathfrak u$. If $R$ is a bilinear form on $\mathfrak m$, let $\tr_gR|_{\mathfrak u}$ be the trace of $R|_{\mathfrak u}$ with respect to $g|_{\mathfrak u}$. Define
\begin{align*}
\mathcal M_T(\mathfrak k)=\{g\in\mathcal M(\mathfrak k)\,|\,\tr_gT|_{\mathfrak k\ominus\mathfrak h}=1\}.
\end{align*}
In what follows, we assume $\mathcal M_T(\mathfrak k)$ carries the topology inherited from $\mathcal M(\mathfrak k)$. There is a natural identification between $\mathcal M_T(\mathfrak g)$ and $\mathcal M_T$. We will need the following bounds on $\lambda_-(g)$, $S(g)$ and $\hat S(g)$.

\begin{lemma}\label{lem_smpl_bd}
If $g$ lies in $\mathcal M_T(\mathfrak k)$ and $\mathfrak u$ is a nonzero subspace of $\mathfrak k\ominus\mathfrak h$, then
\begin{align*}
\lambda_-(g)\ge\omega(\mathfrak k\ominus\mathfrak h)\lambda_-(T|_{\mathfrak k\ominus\mathfrak h}),\qquad \hat S(g|_{\mathfrak u})\le S(g|_{\mathfrak u})\le-\frac12\tr_gB|_{\mathfrak u}\le-\frac{\lambda_-(B)}{2\lambda_-(T)}.
\end{align*}
\end{lemma}

\begin{proof}
We may assume without loss of generality that the decomposition~\eqref{m_decomp} satisfies~\eqref{g_diag_decm};
cf.~\cite[page~180]{MWWZ86}. Let $q$ be a number in $J_{\mathfrak k}$ such that
\begin{align*}
\lambda_-(g)=\min\{x_i\,|\,i\in J_{\mathfrak k}\}=x_q.
\end{align*}
Fix a $Q$-orthonormal basis $(e_j)_{j=1}^{d_q}$ of $\mathfrak m_q$. The inclusions $g\in\mathcal M_T(\mathfrak k)$ and $T\in\mathcal M$ imply
\begin{align*}
1=\tr_gT|_{\mathfrak k\ominus\mathfrak h}\ge\tr_gT|_{\mathfrak m_q}=\sum_{j=1}^{d_q}\frac{T(e_j,e_j)}{g(e_j,e_j)}\ge\frac{d_q\lambda_-(T|_{\mathfrak m_q})}{\lambda_-(g)}\ge\frac{\omega(\mathfrak k\ominus\mathfrak h)\lambda_-(T|_{\mathfrak k\ominus\mathfrak h})}{\lambda_-(g)}.
\end{align*}
Thus, the first estimate must hold.

It is obvious that $\hat S(g|_{\mathfrak u})\le S(g|_{\mathfrak u})$. By formula~\eqref{scal_def}, 
\begin{align*}
S(g|_{\mathfrak u})\le-\frac12\tr_gB|_{\mathfrak u}\le-\frac12\lambda_-(B|_{\mathfrak u})\tr_gQ|_{\mathfrak u}.
\end{align*}
The inclusion $T\in\mathcal M$ implies
\begin{align*}
1\ge\tr_gT|_{\mathfrak u}\ge\lambda_-(T|_{\mathfrak u})\tr_gQ|_{\mathfrak u}.
\end{align*}
Therefore,
\begin{align*}
-\frac12\lambda_-(B|_{\mathfrak u})\tr_gQ|_{\mathfrak u}\le-\frac{\lambda_-(B|_{\mathfrak u})}{2\lambda_-(T|_{\mathfrak u})}\le-\frac{\lambda_-(B)}{2\lambda_-(T)}.
\end{align*}
\end{proof}

We will also need the following simple consequence of~\eqref{Casimir}.

\begin{lemma}\label{lem_sup_pos}
The quantity
\begin{align*}
\sup\big\{\hat S(h)\,\big|\,h\in\mathcal M_T(\mathfrak k)\big\}
\end{align*}
is non-negative.
\end{lemma}

\begin{proof}
Denote $\psi=\tr_QT|_{\mathfrak k\ominus\mathfrak h}$. Because
\begin{align*}
\tr_{\psi Q}T|_{\mathfrak k\ominus\mathfrak h}=\frac1\psi\tr_QT|_{\mathfrak k\ominus\mathfrak h}=1,
\end{align*}
the tensor $\psi Q|_{\mathfrak k\ominus\mathfrak h}$ lies in $\mathcal M_T(\mathfrak k)$.
Using Lemma~\ref{lem_hat_comp} and formula~\eqref{Casimir}, we obtain
\begin{align*}
\sup\big\{\hat S(h)\,\big|\,h\in\mathcal M_T(\mathfrak k)\big\}
\ge\hat S(\psi Q|_{\mathfrak k\ominus\mathfrak h})&=\frac1{2\psi}\sum_{i\in J_{\mathfrak k}}\bigg(d_ib_i-\sum_{j,k\in J_{\mathfrak j}}[ijk]-\frac12\sum_{j,k\in J_{\mathfrak k}}[ijk]\bigg)
\\ &=\frac1{2\psi}\sum_{i\in J_{\mathfrak k}}\bigg(2d_i\zeta_i+\frac12\sum_{j,k\in J_{\mathfrak k}}[ijk]\bigg)\ge0.
\end{align*}
\end{proof}

Let us conclude this subsection with one more auxiliary result about scalar products from $\mathcal M_T(\mathfrak k)$.

\begin{lemma}\label{lem_compact}
Given $\tau>0$, the set
\begin{align*}
\mathcal C(\mathfrak k,\tau)=\{g\in\mathcal M_T(\mathfrak k)\,|\,\lambda_+(g)\le\tau\}
\end{align*}
is compact in $\mathcal M_T(\mathfrak k)$.
\end{lemma}

\begin{proof}
Lemma~\ref{lem_smpl_bd} yields the inclusion
\begin{align*}
\mathcal C(\mathfrak k,\tau)\subset\mathcal D(\mathfrak k,\tau)=\{g\in\mathcal M(\mathfrak k)\,|\,\omega(\mathfrak k\ominus\mathfrak h)\lambda_-(T|_{\mathfrak k\ominus\mathfrak h})\le\lambda_-(g)\le\lambda_+(g)\le\tau\}.
\end{align*}
Exploiting the fact that the set of $k\times k$ matrices with eigenvalues in some bounded closed interval is compact in $\mathbb R^{k^2}$ for $k\ge1$, one can easily verify that $\mathcal D(\mathfrak k,\tau)$ is compact in $\mathcal M(\mathfrak k)$. It is clear that $\mathcal C(\mathfrak k,\tau)$ is closed in $\mathcal M(\mathfrak k)$. Therefore, $\mathcal C(\mathfrak k,\tau)$ must be compact in $\mathcal M(\mathfrak k)$. The assertion of the lemma now follows from the fact that the topology of $\mathcal M_T(\mathfrak k)$ is inherited from $\mathcal M(\mathfrak k)$.
\end{proof}

\subsection{The key estimate}\label{sec_finalise}

Throughout Subsections~\ref{sec_finalise}--\ref{subsec_ex_glob_max}, we suppose $\mathfrak k$ is a Lie subalgebra of $\mathfrak g$ containing $\mathfrak h$ as a proper subset. Recall that, by assumption, $\mathfrak k$ must meet the requirements of Hypothesis~\ref{hyp_flag}. Let $\mathfrak k_1,\ldots,\mathfrak k_r$ be all the maximal Lie subalgebras of $\mathfrak k$ containing $\mathfrak h$ as a proper subset. In Subsection~\ref{sec_finalise}, we suppose that at least one such subalgebra exists. The fact that there are only finitely many follows from Corollary~\ref{cor_finite_subalg}. It is clear that
\begin{align}\label{simple_i_chain}
\mathfrak g\supset\mathfrak k\supset\mathfrak k_i\supset\mathfrak h
\end{align}
is a simple chain for every $i=1,\ldots,r$.

Our first main objective in this subsection is to estimate the values of the functional $\hat S$ on $\mathcal M_T(\mathfrak k)$ in terms of its values on $\mathcal M_T(\mathfrak k_1),\ldots,\mathcal M_T(\mathfrak k_r)$. We achieve this objective in Lemma~\ref{lem_eps+max}. Afterwards, we use the obtained result to show that $\hat S$ has a global maximum on $\mathcal M_T(\mathfrak k)$ if it has global maxima on $\mathcal M_T(\mathfrak k_1),\ldots,\mathcal M_T(\mathfrak k_r)$ and the conditions of Theorem~\ref{thm_PRC} are satisfied. This is the content of Lemma~\ref{lem_step}. It will be convenient for us to denote
\begin{align*}
\mathfrak l_i=\mathfrak k\ominus\mathfrak k_i.
\end{align*}
Let $\Theta(\mathfrak k)$ be the class of $\Ad(H)$-invariant proper subspaces $\mathfrak u\subset\mathfrak k\ominus\mathfrak h$ such that 
\begin{align*}
\mathfrak u\cap\mathfrak l_i\ne\{0\}
\end{align*}
for each $i=1,\ldots,r$. Observe that $\mathfrak u\oplus\mathfrak h$ cannot be a Lie subalgebra of $\mathfrak k$ if $\mathfrak u\in\Theta(\mathfrak k)$.

The following result will help us estimate~$\hat S$. Roughly speaking, it is a consequence of the compactness of the set of decompositions of the form~\eqref{m_decomp}.

\begin{lemma}\label{lem_inf_bdd}
The number
\begin{align*}
\theta=
\begin{cases}
\inf\{\langle\mathfrak u\mathfrak u\mathfrak q\rangle\,|\,\mathfrak u\in\Theta(\mathfrak k)~\rm{and}~\mathfrak q=\mathfrak k\ominus(\mathfrak u\oplus\mathfrak h)\} & \mathrm{if}~\Theta(\mathfrak u)\ne\emptyset \\
1& \mathrm{if}~\Theta(\mathfrak u)=\emptyset
\end{cases}
\end{align*}
is greater than 0.
\end{lemma}

\begin{proof}
Assume the contrary. Then there exists a sequence $(\mathfrak u_j)_{j=1}^\infty\subset\Theta(\mathfrak k)$ such that
\begin{align}\label{eq_aux_lim_spaces}
\lim_{j\to\infty}\langle\mathfrak u_j\mathfrak u_j\mathfrak q_j\rangle=0,\qquad
\mathfrak q_j=\mathfrak k\ominus(\mathfrak u_j\oplus\mathfrak h).
\end{align}
The inclusion $(\mathfrak u_j)_{j=1}^\infty\subset\Theta(\mathfrak k)$ implies
\begin{align}\label{eq_aux_inters_noemp_j}
\mathfrak u_j\cap\mathfrak l_i\ne\{0\},\qquad j\in\mathbb N,~i=1,\ldots,r.
\end{align}
Replacing $(\mathfrak u_j)_{j=1}^\infty$ with a subsequence if necessary, we may assume that the dimension of $\mathfrak u_j$ is independent of~$j$. We denote this dimension by~$m$.

For every $j\in\mathbb N$, choose a $Q$-orthonormal basis $\mathcal E_j=(e_k^j)_{k=1}^m$ of the space $\mathfrak u_j$. The sequence $(\mathcal E_j)_{j=1}^\infty$ has a subsequence converging in $(\mathfrak k\ominus\mathfrak h)^m$ to some 
\begin{align*}
\mathcal E_\infty=(e_k^\infty)_{k=1}^m\in(\mathfrak k\ominus\mathfrak h)^m.
\end{align*}
Let $\mathfrak u_\infty$ be the linear span of~$\mathcal E_\infty$. One can verify that $\mathfrak u_\infty$ is $\Ad(H)$-invariant. Formula~\eqref{eq_aux_lim_spaces} implies
\begin{align*}
\langle\mathfrak u_\infty\mathfrak u_\infty\mathfrak q_\infty\rangle=0,\qquad
\mathfrak q_\infty=\mathfrak k\ominus(\mathfrak u_\infty\oplus\mathfrak h).
\end{align*}
Consequently, $\mathfrak u_\infty\oplus\mathfrak h$ must be a Lie subalgebra of~$\mathfrak k$. Because $(\mathfrak u_j)_{j=1}^\infty\subset\Theta(\mathfrak k)$,
\begin{align*}
\dim\mathfrak k\ominus\mathfrak h>\dim\mathfrak u_j=m=\dim\mathfrak u_\infty,\qquad j\in\mathbb N.
\end{align*}
Therefore, $\mathfrak u_\infty\oplus\mathfrak h$ is a \emph{proper} Lie subalgebra of~$\mathfrak k$. We conclude that $\mathfrak u_\infty\oplus\mathfrak h$ is contained in $\mathfrak k_i$ for some $i=1,\ldots,r$. Our next step is to show that this is impossible. The contradiction will complete the proof.

For every $j\in\mathbb N$, formula~\eqref{eq_aux_inters_noemp_j} yields the existence of a vector 
\begin{align*}
X_j\in\mathfrak u_j\cap\mathfrak l_1
\end{align*}
with $Q(X_j,X_j)=1$. The sequence $(X_j)_{j=1}^\infty$ has a subsequence converging to some~$X_\infty$ in~$\mathfrak k$. It is clear that
\begin{align*}
X_\infty\in\mathfrak u_\infty\cap\mathfrak l_1
\end{align*}
and $Q(X_\infty,X_\infty)=1$. Thus, $\mathfrak u_\infty$ is not contained in $\mathfrak k_1$. Similar arguments show that $\mathfrak u_\infty$ is not in $\mathfrak k_i$ for~$i=2,\ldots,r$.
\end{proof}

Our next result involves the sets $J_{\mathfrak k}$ and $\mathcal C(\mathfrak k,\tau)$ given by Lemmas~\ref{lem_k=m} and~\ref{lem_compact}. We also need the function $\alpha:(0,\infty)\to(0,\infty)$ defined by the formula
\begin{align}\label{def_alpha}
\alpha(\epsilon)=\bigg(\max\bigg\{1,-\frac{2s\lambda_-(B)}{\theta\lambda_-(T)}\epsilon\bigg\}\bigg)^{2^s-1}\epsilon,\qquad \epsilon>0,
\end{align}
where $s$ is the number of summands in~\eqref{m_decomp}.

\begin{lemma}\label{lem_new_two_dist_est}
Let the scalar product $g\in\mathcal M_T(\mathfrak k)$ and the decomposition~(\ref{m_decomp}) satisfy~(\ref{g_diag_decm}). Suppose $\mathcal J$ is a subset of $J_{\mathfrak k}$ such that the space
\begin{align*}
\mathfrak m_{\mathcal J}=\bigoplus_{u\in\mathcal J}\mathfrak m_u
\end{align*}
lies in $\Theta(\mathfrak k)$. Given $\epsilon>0$, assume $\lambda_+(g|_{\mathfrak m_{\mathcal J}})<\epsilon$ and $\hat S(g)>0$. Then $g$ lies in~$\mathcal C(\mathfrak k,\alpha(\epsilon))$.
\end{lemma}

\begin{proof}
The inclusion $\mathfrak m_{\mathcal J}\in\Theta(\mathfrak k)$, Lemma~\ref{lem_inf_bdd} and formula~\eqref{<ijk>[ijk]} imply
\begin{align*}
\sum_{u,v\in\mathcal J}\sum_{w\in J_{\mathfrak k}\setminus\mathcal J}[uvw]\ge\theta>0.
\end{align*}
Consequently, there exists $i\in J_{\mathfrak k}\setminus\mathcal J$ such that
\begin{align*}
\sum_{u,v\in\mathcal J}[uvi]\ge\frac{\theta}{|J_{\mathfrak k}\setminus\mathcal J|}>\frac{\theta}s.
\end{align*}
According to Lemmas~\ref{lem_sc_comp} and~\ref{lem_smpl_bd},
\begin{align*}
\hat S(g)\le S(g)&\le-\frac12\tr_gB|_{\mathfrak k}-\frac14\sum_{u,v,q\in J_{\mathfrak k}}[uvq]\frac{x_q}{x_ux_v} \le
-\frac{\lambda_-(B)}{2\lambda_-(T)}-\frac14\sum_{u,v\in\mathcal J}[uvi]\frac{x_i}{x_ux_v}.
\end{align*}
Since
\begin{align}\label{est_aux_lambda_two}
\max_{u\in\mathcal J}x_u=\lambda_+(g|_{\mathfrak m_\mathcal J})<\epsilon
\end{align}
and $\hat S(g)>0$, the formula
\begin{align}\label{est_aux_xk}
x_i&\le
-4\frac{(\max_{u\in\mathcal J}x_u)^2}{\sum_{u,v\in\mathcal J}[uvi]}\bigg(\hat S(g)+\frac{\lambda_-(B)}{2\lambda_-(T)}\bigg) \notag \\
&<-\frac2{\sum_{u,v\in\mathcal J}[uvi]}\frac{\lambda_-(B)}{\lambda_-(T)}\epsilon^2<-\frac{2s\lambda_-(B)}{\theta\lambda_-(T)}\epsilon^2\le\max\bigg\{1,-\frac{2s\lambda_-(B)}{\theta\lambda_-(T)}\epsilon\bigg\}\epsilon
\end{align}
holds. Suppose $\mathfrak m_i\oplus\mathfrak m_{\mathcal J}\oplus\mathfrak h$ coincides with $\mathfrak k$. In this case,
\begin{align*}
\lambda_+(g)&=\max\Big\{x_i,\max_{u\in\mathcal J}x_u\Big\} \\ &<\max\bigg\{\max\bigg\{1,-\frac{2s\lambda_-(B)}{\theta\lambda_-(T)}\epsilon\bigg\}\epsilon,\epsilon\bigg\} 
\\ &=\max\bigg\{1,-\frac{2s\lambda_-(B)}{\theta\lambda_-(T)}\epsilon\bigg\}\epsilon\le\bigg(\max\bigg\{1,-\frac{2s\lambda_-(B)}{\theta\lambda_-(T)}\epsilon\bigg\}\bigg)^{2^s-1}\epsilon=\alpha(\epsilon).
\end{align*}
Thus, $g$ is in $\mathcal C(\mathfrak k,\alpha(\epsilon))$, and the assertion of the lemma holds.

Suppose $\mathfrak m_i\oplus\mathfrak m_{\mathcal J}\oplus\mathfrak h$ and $\mathfrak k$ are distinct. The inclusion $\mathfrak m_{\mathcal J}\in\Theta(\mathfrak k)$ implies $\mathfrak m_i\oplus\mathfrak m_{\mathcal J}\in\Theta(\mathfrak k)$. Employing Lemma~\ref{lem_inf_bdd} and formula~\eqref{<ijk>[ijk]}, we conclude that
\begin{align*}
\sum_{u,v\in\mathcal J\cup\{i\}}\sum_{w\in J_{\mathfrak k}\setminus(\mathcal J\cup\{i\})} [uvw]\ge\theta>0.
\end{align*}
This means there exists $j\in J_{\mathfrak k}\setminus(\mathcal J\cup\{i\})$ such that
\begin{align*}
\sum_{u,v\in \mathcal J\cup\{i\}}[uvj]\ge\frac{\theta}{|J_{\mathfrak k}\setminus(\mathcal J\cup\{i\})|}>\frac{\theta}s.
\end{align*}
Lemmas~\ref{lem_sc_comp} and~\ref{lem_smpl_bd} imply
\begin{align*}
\hat S(g)&\le
-\frac{\lambda_-(B)}{2\lambda_-(T)}-\frac14\sum_{u,v\in\mathcal J\cup\{i\}}[uvj]\frac{x_j}{x_ux_v}.
\end{align*}
In light of~\eqref{est_aux_lambda_two}, \eqref{est_aux_xk} and the assumption $\hat S(g)>0$, we conclude that
\begin{align*}
x_j<-\frac{2s\lambda_-(B)}{\theta\lambda_-(T)}\bigg(\max\bigg\{1,-\frac{2s\lambda_-(B)}{\theta\lambda_-(T)}\epsilon\bigg\}\epsilon\bigg)^2\le\bigg(\max\bigg\{1,-\frac{2s\lambda_-(B)}{\theta\lambda_-(T)}\epsilon\bigg\}\bigg)^3\epsilon.
\end{align*}
Let $\mathfrak m_i\oplus\mathfrak m_j\oplus\mathfrak m_{\mathcal J}\oplus\mathfrak h$ equal $\mathfrak k$. Then $s$ is no less than $|\mathcal J|+2>2$, and
\begin{align*}
\lambda_+(g)&=\max\Big\{x_i,x_j,\max_{u\in\mathcal J}x_u\Big\} \\ &<\max\bigg\{\max\bigg\{1,-\frac{2s\lambda_-(B)}{\theta\lambda_-(T)}\epsilon\bigg\}\epsilon,\bigg(\max\bigg\{1,-\frac{2s\lambda_-(B)}{\theta\lambda_-(T)}\epsilon\bigg\}\bigg)^3\epsilon,\epsilon\bigg\} 
\\ &=\bigg(\max\bigg\{1,-\frac{2s\lambda_-(B)}{\theta\lambda_-(T)}\epsilon\bigg\}\bigg)^3\epsilon\le\bigg(\max\bigg\{1,-\frac{2s\lambda_-(B)}{\theta\lambda_-(T)}\epsilon\bigg\}\bigg)^{2^s-1}\epsilon=\alpha(\epsilon).
\end{align*}
Thus, the assertion of the lemma holds.

Suppose $\mathfrak m_i\oplus\mathfrak m_j\oplus\mathfrak m_{\mathcal J}\oplus\mathfrak h$ and $\mathfrak k$ are distinct. The inclusion $\mathfrak m_{\mathcal J}\in\Theta(\mathfrak k)$ shows that $\mathfrak m_i\oplus\mathfrak m_j\oplus\mathfrak m_{\mathcal J}\in\Theta(\mathfrak k)$. Continuing to argue as above, we demonstrate that 
\begin{align*}
\lambda_+(g)&<\bigg(\max\bigg\{1,-\frac{2s\lambda_-(B)}{\theta\lambda_-(T)}\epsilon\bigg\}\bigg)^{2^{|J_{\mathfrak k}\setminus\mathcal J|}-1}\epsilon\le\bigg(\max\bigg\{1,-\frac{2s\lambda_-(B)}{\theta\lambda_-(T)}\epsilon\bigg\}\bigg)^{2^s-1}\epsilon=\alpha(\epsilon).
\end{align*}
This completes the proof.
\end{proof}

Denote
\begin{align*}
\mathfrak n_i=\mathfrak k_i\ominus\mathfrak h, \qquad i=1,\ldots,r.
\end{align*}
Lemma~\ref{lem_k=m} implies the existence of sets $J_{\mathfrak k_1},\ldots,J_{\mathfrak k_r}$ such that
\begin{align*}
\mathfrak n_i=\bigoplus_{j\in J_{\mathfrak k_i}}{\mathfrak m}_j,\qquad i=1,\ldots,r.
\end{align*}
It will be convenient for us to define
\begin{align*}
J_{\mathfrak l_i}=J_{\mathfrak k}\setminus J_{\mathfrak k_i},\qquad i=1,\ldots,r.
\end{align*}
Our next result shows that, roughly speaking, a scalar product $g\in\mathcal M_T(\mathfrak k)\setminus\mathcal C(\mathfrak k,\alpha(\epsilon))$ satisfying $\hat S(g)>0$ must be ``large" outside of $\mathfrak k_i$ for some $i=1,\ldots,r$. This result is an important ingredient in the proof of our key estimate for~$\hat S$.

\begin{lemma}\label{lem_index_incl}
Given $\epsilon>0$, consider $g\in\mathcal M_T(\mathfrak k) \setminus\mathcal C(\mathfrak k,\alpha(\epsilon))$ such that $\hat S(g)>0$. Assume the decomposition~(\ref{m_decomp}) satisfies~(\ref{g_diag_decm}). Then the set
\begin{align*}
\mathcal I(g,\epsilon)=\{j\in J_{\mathfrak k}\,|\,x_j<\epsilon\}
\end{align*}
is contained in $J_{\mathfrak k_i}$ for some $i=1,\ldots,r$.
\end{lemma}

\begin{proof}
Denote
\begin{align*}
\mathfrak m_{\mathcal I(g,\epsilon)}=\bigoplus_{j\in\mathcal I(g,\epsilon)}\mathfrak m_j.
\end{align*}
It is clear that 
\begin{align*}
\lambda_+(g|_{\mathfrak m_{\mathcal I(g,\epsilon)}})=\max_{j\in\mathcal I(g,\epsilon)}x_j<\epsilon.
\end{align*}
By assumption, $\hat S(g)$ is positive. The inclusion $g\in\mathcal M_T(\mathfrak k)\setminus\mathcal C(\mathfrak k,\alpha(\epsilon))$ and Lemma~\ref{lem_new_two_dist_est} imply that $\mathfrak m_{\mathcal I(g,\epsilon)}$ does not lie in~$\Theta(\mathfrak k)$. Therefore, either $\mathfrak m_{\mathcal I(g,\epsilon)}$ coincides with $\mathfrak k\ominus\mathfrak h$ or there exists $i=1,\ldots,r$ such that
\begin{align}\label{eq_mI_conta}
\mathfrak m_{\mathcal I(g,\epsilon)}\cap\mathfrak l_i=\{0\}.
\end{align}
In the former case, $\mathcal I(g,\epsilon)$ must equal $J_{\mathfrak k}$, and
\begin{align*}
\lambda_+(g)=\max_{j\in J_{\mathfrak k}}x_j=\max_{j\in\mathcal I(g,\epsilon)}x_j<\epsilon.
\end{align*}
On the other hand, the inclusion $g\in\mathcal M_T(\mathfrak k)\setminus\mathcal C(\mathfrak k,\alpha(\epsilon))$ yields
\begin{align*}
\lambda_+(g)>\alpha(\epsilon)\ge\epsilon.
\end{align*}
Thus, $\mathfrak m_{\mathcal I(g,\epsilon)}$ cannot coincide with $\mathfrak k\ominus\mathfrak h$. We conclude that there exists $i=1,\ldots,r$ satisfying~\eqref{eq_mI_conta}. For any such~$i$, the intersection $\mathcal I(g,\epsilon)\cap J_{\mathfrak l_i}$ is empty, which means $\mathcal I(g,\epsilon)\subset J_{\mathfrak k_i}$.
\end{proof}

Define functions $\beta:(0,\infty)\to(0,\infty)$ and $\kappa:(0,\infty)\to(0,\infty)$ by setting
\begin{align*}
\beta(\epsilon)=-\frac{n\lambda_-(B)-1}{2\epsilon},\qquad \kappa(\epsilon)=\alpha(\beta(\epsilon)),\qquad \epsilon>0,
\end{align*}
where $n$ is the dimension of $M$ and $\alpha(\cdot)$ is given by~\eqref{def_alpha}. We are now ready to state our key estimate on~$\hat S$.

\begin{lemma}\label{lem_eps+max}
Given $\epsilon>0$, the formula
\begin{align}\label{est_eps+max}
\hat S(g)\le\epsilon+\max_{i=1,\ldots,r}\sup\big\{\hat S(h)\,\big|\,h\in\mathcal M_T(\mathfrak k_i)\big\}
\end{align}
holds for every $g\in\mathcal M_T(\mathfrak k)\setminus\mathcal C(\mathfrak k,\kappa(\epsilon))$.
\end{lemma}

\begin{remark}
Lemma~\ref{lem_smpl_bd} implies that the set
\begin{align*}
\big\{\hat S(h)\,\big|\,h\in\mathcal M_T(\mathfrak k_i)\big\}
\end{align*}
is bounded above for every $i=1,\ldots,r$. Therefore, 
the quantity on the right-hand of~\eqref{est_eps+max} is always finite.
\end{remark}

\begin{proof}[Proof of Lemma~\ref{lem_eps+max}]
Choose $g\in\mathcal M_T(\mathfrak k)\setminus\mathcal C(\mathfrak k,\kappa(\epsilon))$. We will show that~\eqref{est_eps+max} holds for~$g$. 
Without loss of generality, suppose the decomposition~\eqref{m_decomp} satisfies~\eqref{g_diag_decm}; cf.~\cite[page~180]{MWWZ86}. If $\hat S(g)\le0$, then~\eqref{est_eps+max} follows from Lemma~\ref{lem_sup_pos}. Thus, we may assume $\hat S(g)>0$. Throughout the remainder of the proof, we fix $i$ with $\mathcal I(g,\beta(\epsilon))\subset J_{\mathfrak k_i}$. Such an $i$ exists by Lemma~\ref{lem_index_incl}. It is clear that $J_{\mathfrak l_i}$ is contained in $J_{\mathfrak k}\setminus\mathcal I(g,\beta(\epsilon))$.
 
According to Lemmas~\ref{lem_elem_id_est} and~\ref{lem_sc_comp},
\begin{align*}
\hat S(g)&\le 
\hat S(g|_{\mathfrak n_i})+S(g|_{\mathfrak l_i}) \\ &\le\hat S(g|_{\mathfrak n_i})+\frac12\sum_{j\in J_{\mathfrak l_i}}\frac{d_jb_j}{x_j} \le\hat S(g|_{\mathfrak n_i})-\frac{\lambda_-(B)}2\sum_{j\in J_{\mathfrak l_i}}\frac{d_j}{x_j}\le \hat S(g|_{\mathfrak n_i})-\frac{n\lambda_-(B)}{2\min_{j\in J_{\mathfrak l_i}}x_j}.
\end{align*}
Recalling the definition of $\mathcal I(g,\beta(\epsilon))$, we find
\begin{align*}
\min_{j\in J_{\mathfrak l_i}}x_j\ge\min_{j\in J_{\mathfrak k}\setminus\mathcal I(g,\beta(\epsilon))}x_j\ge\beta(\epsilon).
\end{align*}
Therefore,
\begin{align*}
\hat S(g)&\le \hat S(g|_{\mathfrak n_i})-\frac{n\lambda_-(B)}{2\beta(\epsilon)}<\hat S(g|_{\mathfrak n_i})+\epsilon.
\end{align*}
Let us show that
\begin{align*}
\hat S(g|_{\mathfrak n_i})\le\sup\big\{\hat S(h)\,\big|\,h\in\mathcal M_T(\mathfrak k_i)\big\}.
\end{align*}
Inequality~\eqref{est_eps+max} will follow immediately. If $\psi_i=\tr_gT|_{\mathfrak n_i}$, then
\begin{align*}
\tr_{\psi_ig}T|_{\mathfrak n_i}=\frac1{\psi_i}\tr_gT|_{\mathfrak n_i}=1,
\end{align*}
which means the scalar product $\psi_ig|_{\mathfrak n_i}$ lies in~$\mathcal M_T(\mathfrak k_i)$. Keeping in mind that $g\in\mathcal M_T(\mathfrak k)$, we estimate
\begin{align*}
\psi_i=\tr_gT|_{\mathfrak n_i}<\tr_gT|_{\mathfrak k\ominus\mathfrak h}=1.
\end{align*}
As a consequence,
\begin{align*}
\hat S(g|_{\mathfrak n_i})=\psi_i\hat S(\psi_ig|_{\mathfrak n_i})\le\psi_i\sup\big\{\hat S(h)\,\big|\,h\in\mathcal M_T(\mathfrak k_i)\big\}<\sup\big\{\hat S(h)\,\big|\,h\in\mathcal M_T(\mathfrak k_i)\big\}.
\end{align*}
\end{proof}

Our goal in Subsection~\ref{subsec_ex_glob_max} will be to show that $\hat S$ has a global maximum on $\mathcal M_T(\mathfrak k)$ under the assumptions of Theorem~\ref{thm_PRC}. We will do so using induction in the dimension of~$\mathfrak k$. The following lemma will help us prove the inductive step. As above, we define $\mathfrak j$ and $J_{\mathfrak j}$ by the first formulas in~\eqref{def_funny_subsp} and~\eqref{def_diff_Js}. It will be convenient for us to set
\begin{align*}
\mathfrak j_i=\mathfrak g\ominus\mathfrak k_i,\qquad J_{\mathfrak j_i}=J_{\mathfrak g}\setminus J_{\mathfrak k_i},\qquad i=1,\ldots,r.
\end{align*}

\begin{lemma}\label{lem_step}
Assume that the following statements are satisfied for each $i=1,\ldots,r$:
\begin{enumerate}
\item
The restriction of $\hat S$ to $\mathcal M_T(\mathfrak k_i)$ has a global maximum.
\item
The inequality
\begin{align*}
\frac{\lambda_-(T|_{\mathfrak n_i})}{\tr_Q T|_{\mathfrak l_i}}>\eta(\mathfrak k,\mathfrak k_i)
\end{align*}
holds.
\end{enumerate}
Then the restriction of $\hat S$ to $\mathcal M_T(\mathfrak k)$ has a global maximum.
\end{lemma}

\begin{proof}
Fix an index $i$ such that
\begin{align*}
\sup\big\{\hat S(h)\,\big|\,h\in\mathcal M_T(\mathfrak k_i)\big\}=\max_{j=1,\ldots,r}\sup\big\{\hat S(h)\,\big|\,h\in\mathcal M_T(\mathfrak k_j)\big\}.
\end{align*}
By hypothesis, there exists $g_0\in\mathcal M_T(\mathfrak k_i)$ satisfying
\begin{align*}
\hat S(g_0)=\sup\big\{\hat S(h)\,\big|\,h\in\mathcal M_T(\mathfrak k_i)\big\}.
\end{align*}
Without loss of generality, suppose the decomposition~\eqref{m_decomp} is such that
\begin{align*}
g_0=\sum_{j\in J_{\mathfrak k_i}}y_j\pi_{\mathfrak m_j}^*Q,\qquad y_j>0.
\end{align*}
Given $t>\tr_QT|_{\mathfrak l_i}$, define $g(t)\in\mathcal M_T(\mathfrak k)$ by the formulas
\begin{align*}
g(t)=\sum_{j\in J_{\mathfrak k_i}}\phi(t)y_j\pi_{\mathfrak m_j}^*Q+\sum_{j\in J_{\mathfrak l_i}}t\pi_{\mathfrak m_j}^*Q,
\qquad
\phi(t)=\frac t{t-\tr_QT|_{\mathfrak l_i}}.
\end{align*}
We will show that $\hat S(g(t))>\hat S(g_0)$ for some~$t$. Together with Lemma~\ref{lem_eps+max}, this will imply the existence of a global maximum of $\hat S$ on $\mathcal M_T(\mathfrak k)$.

Using~\eqref{eq_sc_comp},~\eqref{eq_hat_comp} and the first line in~\eqref{elem_eq}, we compute
\begin{align*}
\lim_{t\to\infty}\hat S(g(t))&=\lim_{t\to\infty}\big(S(g(t)|_{\mathfrak n_i})+S(g(t)|_{\mathfrak l_i})\big) \\
&\hphantom{=}~-\frac12\lim_{t\to\infty}\bigg(\sum_{j\in J_{\mathfrak k_i}}\sum_{k,l\in J_{\mathfrak j}}\frac{[jkl]}{\phi(t)y_j}+\sum_{j\in J_{\mathfrak l_i}}\sum_{k,l\in J_{\mathfrak j}}\frac{[jkl]}t\bigg) 
\\
&\hphantom{=}~-\frac14\lim_{t\to\infty}\sum_{j,k\in J_{\mathfrak l_i}}\sum_{l\in J_{\mathfrak k_i}}[jkl]\bigg(\frac{\phi(t)y_l}{t^2}+\frac2{\phi(t)y_l}\bigg) \\
&=\lim_{t\to\infty}\bigg(\frac{S(g_0)}{\phi(t)}+\frac12\sum_{j\in J_{\mathfrak l_i}}\frac{d_jb_j}t-\frac14\sum_{j,k,l\in J_{\mathfrak l_i}}\frac{[jkl]}t\bigg) \\
&\hphantom{=}~-\frac12\sum_{j\in J_{\mathfrak k_i}}\sum_{k,l\in J_{\mathfrak j}}\frac{[jkl]}{y_j}-\frac12\sum_{j,k\in J_{\mathfrak l_i}}\sum_{l\in J_{\mathfrak k_i}}\frac{[jkl]}{y_l} \\
&=S(g_0)-\frac12\sum_{j\in J_{\mathfrak k_i}}\sum_{k,l\in J_{\mathfrak j_i}}\frac{[jkl]}{y_j}=\hat S(g_0).
\end{align*}
To prove that $\hat S(g(t))>\hat S(g_0)$ for some~$t$, it suffices to demonstrate that $\frac d{dt}\hat S(g(t))<0$ when $t$ is large. Observe that
\begin{align*}
\frac d{dt}\phi(t)=-\frac{\tr_QT|_{\mathfrak l_i}}{(t-\tr_QT|_{\mathfrak l_i})^2},\qquad \frac d{dt}\frac1{\phi(t)}=\frac{\tr_QT|_{\mathfrak l_i}}{t^2},\qquad \frac d{dt}\frac{\phi(t)}{t^2}=-\frac{2t-\tr_QT|_{\mathfrak l_i}}{(t^2-t\tr_QT|_{\mathfrak l_i})^2}.
\end{align*}
Computing as above and utilising~\eqref{trace_bd},~\eqref{<ijk>[ijk]} and Lemma~\ref{lem_vanish_brak}, we obtain
\begin{align*}
\frac d{dt}\hat S(g(t))&=\frac d{dt}\bigg(\frac{S(g_0)}{\phi(t)}+\frac12\sum_{j\in J_{\mathfrak l_i}}\frac{d_jb_j}t-\frac14\sum_{j,k,l\in J_{\mathfrak l_i}}\frac{[jkl]}t
\bigg) \\
&\hphantom{=}~-
\frac12\frac d{dt}\bigg(\sum_{j\in J_{\mathfrak k_i}}\sum_{k,l\in J_{\mathfrak j}}\frac{[jkl]}{\phi(t)y_j}+\sum_{j\in J_{\mathfrak l_i}}\sum_{k,l\in J_{\mathfrak j}}\frac{[jkl]}t\bigg) 
\\
&\hphantom{=}~-\frac14\frac d{dt}\sum_{j,k\in J_{\mathfrak l_i}}\sum_{l\in J_{\mathfrak k_i}}[jkl]\bigg(\frac{\phi(t)y_l}{t^2}+\frac2{\phi(t)y_l}\bigg)
\\
&=\frac{S(g_0)\tr_QT|_{\mathfrak l_i}}{t^2}+\frac{\tr_QB|_{\mathfrak l_i}}{2t^2}+\frac{\langle\mathfrak l_i\mathfrak l_i\mathfrak l_i\rangle}{4t^2} -\frac{\tr_QT|_{\mathfrak l_i}}{2t^2}\bigg(\sum_{j\in J_{\mathfrak k_i}}\sum_{k,l\in J_{\mathfrak j}}\frac{[jkl]}{y_j}\bigg)
\\
&\hphantom{=}~+\frac1{2t^2}\langle \mathfrak l_i\mathfrak j\mathfrak j\rangle -\frac14\sum_{j,k\in J_{\mathfrak l_i}}\sum_{l\in J_{\mathfrak k_i}}[jkl]\bigg(-\frac{2t-\tr_QT|_{\mathfrak l_i}}{(t^2-t\tr_QT|_{\mathfrak l_i})^2}y_l+\frac{2\tr_QT|_{\mathfrak l_i}}{t^2y_l}\bigg)
\\
&=
\frac{\hat S(g_0)\tr_QT|_{\mathfrak l_i}}{t^2}+\frac{\tr_QB|_{\mathfrak l_i}}{2t^2}+
\frac{\langle \mathfrak l_i\mathfrak l_i\mathfrak l_i\rangle}{4t^2}+\frac1{2t^2}\langle\mathfrak l_i\mathfrak j\mathfrak j\rangle
\\
&\hphantom{=}~+\frac{2t-\tr_QT|_{\mathfrak l_i}}{4(t^2-t\tr_QT|_{\mathfrak l_i})^2}\sum_{j,k\in J_{\mathfrak l_i}}\sum_{l\in J_{\mathfrak k_i}}[jkl]y_l.
\end{align*}
It is obvious that $\frac d{dt}\hat S(g(t))<0$ if and only if $t^2\frac d{dt}\hat S(g(t))<0$. Thus, to prove that $\frac d{dt}\hat S(g(t))<0$ for large~$t$, it suffices to show that
\begin{align}\label{lim_zero}
\lim_{t\to\infty}t^2\frac d{dt}\hat S(g(t))<0.
\end{align}
Using the above expression for $\frac d{dt}\hat S(g(t))$, we calculate
\begin{align*}
4\lim_{t\to\infty}t^2\frac d{dt}\hat S(g(t))&=
4\hat S(g_0)\tr_QT|_{\mathfrak l_i}+2\tr_QB|_{\mathfrak l_i}+
\langle\mathfrak l_i\mathfrak l_i\mathfrak l_i\rangle+2\langle \mathfrak l_i\mathfrak j\mathfrak j\rangle.
\end{align*}
Lemmas~\ref{lem_hat_comp} and~\ref{lem_smpl_bd}, along with~\eqref{trace_bd},~\eqref{<ijk>[ijk]} and Lemma~\ref{lem_vanish_brak}, imply
\begin{align*}
4\hat S(g_0)&=
2\sum_{j\in J_{\mathfrak k_i}}\frac{d_jb_j}{y_j}-2\sum_{j\in J_{\mathfrak k_i}}\sum_{k,l\in J_{\mathfrak j_i}}\frac{[jkl]}{y_j}-\sum_{j,k,l\in J_{\mathfrak k_i}}[jkl]\frac{y_l}{y_jy_k}
\\
&=
2\sum_{j\in J_{\mathfrak k_i}}\frac{d_jb_j}{y_j}-2\sum_{j\in J_{\mathfrak k_i}}\sum_{k,l\in J_{\mathfrak j_i}}\frac{[jkl]}{y_j}-\frac12\sum_{j,k,l\in J_{\mathfrak k_i}}\frac{[jkl]}{y_j}\Big(\frac{y_l}{y_k}+\frac{y_k}{y_l}\Big)
\\
&\le
\sum_{j\in J_{\mathfrak k_i}}\frac1{y_j}\bigg(2d_jb_j-2\sum_{k,l\in J_{\mathfrak j_i}}[jkl]-\sum_{k,l\in J_{\mathfrak k_i}}[jkl]\bigg)
\\
&\le
\frac1{\lambda_-(g_0)}\bigg(2\sum_{j\in J_{\mathfrak k_i}}d_jb_j-2\sum_{j\in J_{\mathfrak k_i}}\sum_{k,l\in J_{\mathfrak j_i}}[jkl]-\sum_{j,k,l\in J_{\mathfrak k_i}}[jkl]\bigg)
\\ 
&\le\frac{-2\tr_QB|_{\mathfrak n_i}-2\langle\mathfrak n_i\mathfrak j_i\mathfrak j_i\rangle-\langle\mathfrak n_i\mathfrak n_i\mathfrak n_i\rangle}{\omega(\mathfrak n_i)\lambda_-(T|_{\mathfrak n_i})}.
\end{align*}
(The penultimate estimate exploits the formula
\begin{align*}
2d_jb_j-2\sum_{k,l\in J_{\mathfrak j_i}}[jkl]-\sum_{k,l\in J_{\mathfrak k_i}}[jkl]=4d_j\zeta_j+\sum_{k,l\in J_{\mathfrak k_i}}[jkl]\ge0,\qquad j\in J_{\mathfrak k_i},
\end{align*}
a consequence of~\eqref{Casimir}.) Therefore, to prove~\eqref{lim_zero}, it suffices to show that
\begin{align*}
\frac{-2\tr_QB|_{\mathfrak n_i}-2\langle\mathfrak n_i\mathfrak j_i\mathfrak j_i\rangle-\langle\mathfrak n_i\mathfrak n_i\mathfrak n_i\rangle}{\omega(\mathfrak n_i)\lambda_-(T|_{\mathfrak n_i})}\tr_QT|_{\mathfrak l_i}+2\tr_QB|_{\mathfrak l_i}+
\langle\mathfrak l_i\mathfrak l_i\mathfrak l_i\rangle+2\langle\mathfrak l_i\mathfrak j\mathfrak j\rangle<0.
\end{align*}
After some elementary transformations, this becomes
\begin{align*}
\frac{\lambda_-(T|_{\mathfrak n_i})}{\tr_QT|_{\mathfrak l_i}}>\frac{2\tr_QB|_{\mathfrak n_i}+2\langle\mathfrak n_i\mathfrak j_i\mathfrak j_i\rangle+\langle\mathfrak n_i\mathfrak n_i\mathfrak n_i\rangle}{\omega(\mathfrak n_i)(2\tr_QB|_{\mathfrak l_i}+\langle\mathfrak l_i\mathfrak l_i\mathfrak l_i\rangle+2\langle\mathfrak l_i\mathfrak j\mathfrak j\rangle)}=\eta(\mathfrak k,\mathfrak k_i),
\end{align*}
which is satisfied by hypothesis. Thus,~\eqref{lim_zero} holds, and $\frac d{dt}\hat S(g(t))<0$ for large~$t$. It is easy to establish the existence of $t_0>\tr_QT|_{\mathfrak l_i}$ such that
\begin{align*}
\hat S(g(t_0))>\lim_{t\to\infty}\hat S(g(t))=\hat S(g_0).
\end{align*}
Applying Lemma~\ref{lem_eps+max} with
\begin{align*}
\epsilon=\frac{\hat S(g(t_0))-\hat S(g_0)}2>0,
\end{align*}
we conclude that
$\hat S(g)<\hat S(g(t_0))$ for all $g\in\mathcal M_T(\mathfrak k)\setminus\mathcal C(\mathfrak k,\kappa(\epsilon))$. To complete the proof, we need to demonstrate that $\hat S$ has a global maximum on $\mathcal C(\mathfrak k,\kappa(\epsilon))$. However, this is an immediate consequence of the compactness of~$\mathcal C(\mathfrak k,\kappa(\epsilon))$.
\end{proof}

\begin{remark}\label{rem_proper}
The proof of the lemma shows that $\hat S(g(t))$ converges to $\hat S(g_0)$ as $t$ goes to infinity. Therefore, the inclusion
\begin{align*}
\hat S(\{g(t)\,|\,t\ge2\tr_QT|_{\mathfrak l_i}\})\subset\big[\hat S(g_0)-\sigma,\hat S(g_0)+\sigma\big]
\end{align*} 
holds for some $\sigma>0$. We conclude that the preimage of the interval $\big[\hat S(g_0)-\sigma,\hat S(g_0)+\sigma\big]$ under $\hat S$ has a non-compact intersection with $\mathcal M_T(\mathfrak k)$. This means the restriction of $\hat S$ to $\mathcal M_T(\mathfrak k)$ cannot be proper.
\end{remark}

\subsection{The existence of global maxima}\label{subsec_ex_glob_max}

As in Subsection~\ref{sec_finalise}, suppose $\mathfrak k$ is a Lie subalgebra of $\mathfrak g$ containing $\mathfrak h$ as a proper subset. Recall that $\mathfrak k$ must meet the requirements of Hypothesis~\ref{hyp_flag}. Our next goal is to prove by induction that $\hat S$ has a global maximum on $\mathcal M_T(\mathfrak k)$ under the assumptions of Theorem~\ref{thm_PRC}. The following result will enable us to take the basis step and help with the inductive step.

\begin{lemma}\label{lem_1st_step}
If $\mathfrak h$ is a maximal Lie subalgebra of~$\mathfrak k$, then there exists $g\in\mathcal M_T(\mathfrak k)$ such that $\hat S(g)\ge\hat S(h)$ for all $h\in\mathcal M_T(\mathfrak k)$.
\end{lemma}

\begin{proof}
The formulas
\begin{align*}
\frac1{\lambda_+(h)}\tr_QT|_{\mathfrak k\ominus\mathfrak h}\le\tr_hT|_{\mathfrak k\ominus\mathfrak h}\le\frac1{\lambda_-(h)}\tr_QT|_{\mathfrak k\ominus\mathfrak h}
\end{align*}
and $\tr_hT|_{\mathfrak k\ominus\mathfrak h}=1$ hold whenever $h$ lies in~$\mathcal M_T(\mathfrak k)$. As a consequence,
\begin{align*}
\lambda_-(h)\le\tr_QT|_{\mathfrak k\ominus\mathfrak h}\le\lambda_+(h),\qquad h\in\mathcal M_T(\mathfrak k).
\end{align*}
Applying Lemma~\ref{lem_est_max} with $\tau_1=\tau_2=\tr_QT|_{\mathfrak k\ominus\mathfrak h}$, we find
\begin{align}\label{eq_max_one}
\hat S(h)\le S(h)\le A-D\lambda_+(h)^b,\qquad h\in\mathcal M_T(\mathfrak k),
\end{align}
where the constants $A>0$, $D>0$ and $b>0$ depend only on~$G$, $H$, $\mathfrak k$, $Q$ and $T$. Fix $h_0\in\mathcal M_T(\mathfrak k)$ and suppose
\begin{align*}
\tau=\bigg|\frac{A-\hat S(h_0)}D\bigg|^{\frac1b}+1>0.
\end{align*}
According to Lemma~\ref{lem_compact}, the set $\mathcal C(\mathfrak k,\tau)$ is compact in $\mathcal M_T(\mathfrak k)$. Consequently, there exists $g\in\mathcal C(\mathfrak k,\tau)$ such that $\hat S(g)\ge\hat S(h)$ for all $h\in\mathcal C(\mathfrak k,\tau)$.
Formula~\eqref{eq_max_one} implies that
$\hat S(h_0)>\hat S(h)$ if $h$ lies in $\mathcal M_T(\mathfrak k)\setminus\mathcal C(\mathfrak k,\tau)$. This means $h_0$ is in $\mathcal C(\mathfrak k,\tau)$ and 
\begin{align*}
\hat S(g)\ge\hat S(h_0)>\hat S(h)
\end{align*}
for all $h\in\mathcal M_T(\mathfrak k)\setminus\mathcal C(\mathfrak k,\tau)$. Thus, the global maximum of $\hat S$ on $\mathcal M_T(\mathfrak k)$ exists and is attained at $g$.
\end{proof}

The following result concludes our analysis of~$\hat S$.

\begin{lemma}\label{lem_ind_dim_subgr}
Assume that
\begin{align*}
\frac{\lambda_-(T|_{\mathfrak k''\ominus\mathfrak h})}{\tr_Q T|_{\mathfrak k'\ominus\mathfrak k''}}>\eta(\mathfrak k',\mathfrak k'')
\end{align*}
for every simple chain
\begin{align*}
\mathfrak g\supset\mathfrak k'\supset\mathfrak k''\supset\mathfrak h
\end{align*}
such that $\mathfrak k'\subset\mathfrak k$. Then the restriction of $\hat S$ to $\mathcal M_T(\mathfrak k)$ has a global maximum.
\end{lemma}

\begin{proof}
We proceed by induction. If $\dim\mathfrak k\ominus\mathfrak h$ equals~1, then $\mathfrak h$ must be a maximal Lie subalgebra of $\mathfrak k$. In this case, the existence of~$g\in\mathcal M_T(\mathfrak k)$ such that 
\begin{align}\label{def_glob_max}
\hat S(g)\ge\hat S(h),\qquad h\in\mathcal M_T(\mathfrak k),
\end{align}
follows from Lemma~\ref{lem_1st_step}. This is the basis of induction.

Fix $m=1,\ldots,n-1$, where $n$ is the dimension of~$M$. Suppose $\hat S$ has a global maximum on $\mathcal M_T(\mathfrak s)$ for every Lie subalgebra $\mathfrak s\subset\mathfrak g$ satisfying the formulas
\begin{align*}
\mathfrak h\subset\mathfrak s,\qquad 1\le\dim\mathfrak s\ominus\mathfrak h\le m.
\end{align*}
This is the induction hypothesis.

Let $\dim\mathfrak k\ominus\mathfrak h$ equal $m+1$. If $\mathfrak h$ is a maximal Lie subalgebra of $\mathfrak k$, then the existence of $g\in\mathcal M_T(\mathfrak k)$ satisfying~\eqref{def_glob_max} follows from Lemma~\ref{lem_1st_step}. Suppose $\mathfrak h$ is not. As in Subsection~\ref{sec_finalise}, denote by $\mathfrak k_1,\ldots,\mathfrak k_r$ the maximal Lie subalgebras of $\mathfrak k$ containing $\mathfrak h$ as a proper subset. It is clear that
\begin{align*}
1\le\dim\mathfrak k_i\ominus\mathfrak h\le m,\qquad i=1,\ldots,r.
\end{align*}
By the induction hypothesis, the restriction of $\hat S$ to $\mathfrak k_i$ has a global maximum for each~$i$. The existence of $g\in\mathcal M_T(\mathfrak k)$ satisfying~\eqref{def_glob_max} follows from this fact and Lemma~\ref{lem_step}.
\end{proof}

\begin{remark}\label{rem_milder}
Some of the results in Subsections~\ref{sec_lemmas}--\ref{subsec_ex_glob_max} hold under milder assumptions than those imposed above. In particular, Lemmas~\ref{lem_k=m}, \ref{lem_vanish_brak}, \ref{lem_sc_comp}--\ref{lem_eps+max} and~\ref{lem_1st_step} do not use requirement~2 of Hypothesis~\ref{hyp_flag}.
\end{remark}

\subsection{The completion of the proof of Theorem~\ref{thm_PRC}}\label{sec_final2}

Setting $\mathfrak k=\mathfrak g$ in Lemma~\ref{lem_ind_dim_subgr}, we conclude that the restriction of $\hat S$ to $\mathcal M_T$ has a global maximum. By definition, the maps $\hat S$ and $S$ coincide on $\mathcal M_T$. Ergo, there exists $g\in\mathcal M_T$ such that $S(g)\ge S(h)$ for all $h\in\mathcal M_T$. Lemma~\ref{lem_var} tells us that the Ricci curvature of $g$ equals $cT$ for some $c\in\mathbb R$. To complete the proof of Theorem~\ref{thm_PRC}, we need to show that $c>0$. 

By Bochner's theorem (see~\cite[Theorem~1.84]{AB87}), the space $M$ cannot support a $G$-invariant Riemannian metric with negative-definite Ricci curvature. It follows that $c\ge0$. Let us show that $M$ cannot support a Ricci-flat $G$-invariant metric. This will immediately imply that $c>0$.

We argue by contradiction. Assume there exists a Ricci-flat $G$-invariant metric on $M$. Employing Bochner's theorem again, we conclude that the isometry group of $M$ with respect to this metric must be abelian. It follows that
\begin{align*}
\gamma\gamma'\mu=\gamma'\gamma\mu,\qquad \gamma,\gamma'\in G,~\mu\in M.
\end{align*}
Replacing $\gamma'$ with $\chi\in H$ and choosing $\mu=\gamma^{-1}H$, we obtain
\begin{align*}
\gamma\chi\gamma^{-1}H=H,\qquad \gamma\in G,~\chi\in H.
\end{align*}
This formula implies
\begin{align*}
[\mathfrak m,\mathfrak h]\subset[\mathfrak g,\mathfrak h]\subset\mathfrak h.
\end{align*}
At the same time, $[\mathfrak m,\mathfrak h]$ is contained in $\mathfrak m$ because $\mathfrak m$ is $\Ad(H)$-invariant. Thus, $[\mathfrak m,\mathfrak h]$ is equal to $\{0\}$.

Let us turn our attention to the decomposition~\eqref{m_decomp}. Given $i=1,\ldots,s$, the representation $\Ad(H)|_{\mathfrak m_i}$ is trivial. Its irreducibility implies that $d_i=1$. In light of~\eqref{dimM}, this means~$s\ge3$. The space $\mathfrak m_1\oplus\mathfrak h$ is a Lie subalgebra of $\mathfrak g$ containing $\mathfrak h$ as a proper subset. Clearly,
\begin{align*}
\mathfrak m_1\subset\mathfrak m_1\oplus\mathfrak h,\qquad \mathfrak m_2\subset\mathfrak g\ominus(\mathfrak m_1\oplus\mathfrak h).
\end{align*}
Because the representations $\Ad(H)|_{\mathfrak m_1}$ and $\Ad(H)|_{\mathfrak m_2}$ are both trivial, they must be equivalent. However, this contradicts requirement~1 of Hypothesis~\ref{hyp_flag}.

\subsection{Two corollaries}\label{subsec_cor}

In this subsection, we state and prove two corollaries of Theorem~\ref{thm_PRC}. The first one offers an alternative to~\eqref{ineq_main_thm}.

\begin{corollary}
Suppose Hypothesis~\ref{hyp_flag} is satisfied for $M$. Given $T\in\mathcal M$, if
\begin{align*}
\frac{\lambda_-(T|_{\mathfrak n})}{\lambda_+(T|_{\mathfrak l})}>\eta(\mathfrak k,\mathfrak k')\dim\mathfrak l
\end{align*}
for every simple chain of the form~(\ref{flag}), then there exist $g\in\mathcal M_T$ such that $S(g)\ge S(h)$ for all $h\in\mathcal M_T$. The Ricci curvature of $g$ equals $cT$ with $c>0$.
\end{corollary}

\begin{proof}
The corollary follows from Theorem~\ref{thm_PRC} and the obvious estimate
$\tr_QT|_{\mathfrak l}\le\lambda_+(T|_{\mathfrak l})\dim\mathfrak l$.
\end{proof}

Our next result underlies the discussion of Ricci iterations in Section~\ref{sec_iter}.

\begin{corollary}\label{cor_all}
Suppose the homogeneous space $M$ admits a decomposition of the form~(\ref{m_decomp}) such that the following requirements hold:
\begin{enumerate}
\item
The representation $\Ad(H)|_{\mathfrak m_i}$ is trivial if and only if $i=1$.
\item
The space $\mathfrak m_1\oplus\mathfrak h$ is the only proper Lie subalgebra of~$\mathfrak g$ containing $\mathfrak h$ as a proper subset.
\end{enumerate}
Given $T\in\mathcal M$, there exists $g\in\mathcal M_T$ such that $S(g)\ge S(h)$ for all $h\in\mathcal M_T$. The Ricci curvature of $g$ equals $cT$ for some $c>0$.
\end{corollary}

\begin{proof}
Recalling~\eqref{dimM} and Remark~\ref{rem_1d_alg}, one easily verifies that Hypothesis~\ref{hyp_flag} holds for $M$. Moreover,
\begin{align}\label{simpl_chain_always}
\mathfrak g\supset\mathfrak g\supset\mathfrak m_1\oplus\mathfrak h\supset\mathfrak h
\end{align}
is the only simple chain associated with~$M$. Proposition~\ref{lem_eta0} implies that $\eta(\mathfrak g,\mathfrak m_1\oplus\mathfrak h)=0$. Thus, inequality~\eqref{ineq_main_thm} is necessarily satisfied for~\eqref{simpl_chain_always}. In light of these observations, Theorem~\ref{thm_PRC} yields the result.
\end{proof}

\begin{remark}
The triviality and the irreducibility of $\Ad(H)|_{\mathfrak m_1}$ imply that the dimension $d_1$ of the space~$\mathfrak m_1$ in Corollary~\ref{cor_all} equals~1.
\end{remark}

\begin{remark}
Let Hypothesis~\ref{hyp_flag} be satisfied. Assume~\eqref{ineq_main_thm} holds for every $T\in\mathcal M$ for every simple chain associated with~$M$. Then one can show that $M$ admits a decomposition of the form~\eqref{m_decomp} that meets requirements~1 and~2 of Corollary~\ref{cor_all}. The argument is similar in spirit to the proof of Proposition~\ref{lem_eta0}. We leave the details to the reader.
\end{remark}

Corollary~\ref{cor_all} applies if $M$ coincides with, for instance, $SO(2k)/SU(k)$ for $k\ge3$, $SU(k+l)/SU(k)\times SU(l)$ for $k,l\ge2$, $Sp(k)/SU(k)$ for $k\ge3$ or~$E_7/E_6$ (the corresponding embeddings are given in~\cite[Examples~I.24, II.7, III.8 and~IV.19]{WDMK08}). In all these cases, $M$ has two inequivalent irreducible summands in its isotropy representation. Thus, the existence of~$g\in\mathcal M$ with Ricci curvature $cT$ for some $c>0$ also follows from~\cite[Proposition~3.1]{AP16}. The authors were unable to find examples of $M$ that would satisfy the assumptions of Corollary~\ref{cor_all} and have three or more irreducible summands in their isotropy representations. We hope that such examples will emerge in the future.

\section{The case of two inequivalent irreducible summands}\label{sec_2sum}

Theorem~\ref{thm_PRC} provides a sufficient condition for the existence of a metric $g\in\mathcal M$ whose Ricci curvature equals $cT$ with $c>0$. We will show that this condition is necessary when the isotropy representation of $M$ splits into two inequivalent irreducible summands. Our argument will rely on~\cite[Proposition~3.1]{AP16}.

Suppose $s=2$ in every decomposition of the form~\eqref{m_decomp}, i.e.,
\begin{align*}
\mathfrak m=\mathfrak m_1\oplus\mathfrak m_2.
\end{align*}
Let $\Ad(H)|_{\mathfrak m_1}$ and $\Ad(H)|_{\mathfrak m_2}$ be inequivalent. According to Theorem~\ref{thm_PRC}, finding a metric whose Ricci curvature equals $cT$ for some $c>0$ is always possible if $\mathfrak h$ is maximal in~$\mathfrak g$. Thus, we may assume that there exists a Lie subalgebra $\mathfrak s\subset\mathfrak g$ such that
\begin{align*}
\mathfrak g\supset\mathfrak s\supset\mathfrak h,\qquad \mathfrak h\ne\mathfrak s,\qquad \mathfrak s\ne\mathfrak g.
\end{align*}
It is clear that $\mathfrak s\ominus\mathfrak h$ is a proper $\Ad(H)$-invariant subspace of~$\mathfrak m$. The only such subspaces are $\mathfrak m_1$ and~$\mathfrak m_2$. Therefore, $\mathfrak s$ must equal $\mathfrak m_1\oplus\mathfrak h$ or $\mathfrak m_2\oplus\mathfrak h$. Suppose
\begin{align}\label{eq_k1_wlog}
\mathfrak s=\mathfrak m_1\oplus\mathfrak h.
\end{align}
If $\mathfrak m_2\oplus\mathfrak h$ is a Lie subalgebra of $\mathfrak g$, then $[112]=[221]=0$. In this case, all the metrics in $\mathcal M$ have the same Ricci curvature, and the problem of solving equation~\eqref{eq_PRC} becomes trivial; see, e.g.,~\cite[Subsection~4.2]{APYRsubm}. Thus, we may assume $\mathfrak m_2\oplus\mathfrak h$ is not closed under the Lie bracket. This implies $[221]>0$.

Let us verify Hypothesis~\ref{hyp_flag}. It is clear that $\mathfrak s$ given by~\eqref{eq_k1_wlog} is the unique proper Lie subalgebra of $\mathfrak g$ such that $\mathfrak h\subset\mathfrak s$ and $\mathfrak h\ne\mathfrak s$. The only nonzero $\Ad(H)$-invariant subspace of~$\mathfrak s\ominus\mathfrak h$ is $\mathfrak m_1$, and the only such subspace of~$\mathfrak g\ominus\mathfrak s$ is $\mathfrak m_2$. Since $\Ad(H)|_{\mathfrak m_1}$ and $\Ad(H)|_{\mathfrak m_2}$ are inequivalent, $\mathfrak s$ meets requirement~1 of Hypothesis~\ref{hyp_flag}. If
\begin{align*}
[\mathfrak m_2,\mathfrak s]=\{0\},
\end{align*}
then $[112]=[221]=0$, which contradicts the formula~$[221]>0$. Thus, $\mathfrak s$ meets requirement~2 of Hypothesis~\ref{hyp_flag}.

It is easy to see that
\begin{align*}
\mathfrak g\supset\mathfrak g\supset\mathfrak s\supset\mathfrak h
\end{align*}
is the only simple chain associated with~$M$. Setting $\mathfrak k=\mathfrak g$ and $\mathfrak k'=\mathfrak s$ in~\eqref{flag}, we obtain
\begin{align*}
\mathfrak j=\{0\},\qquad \mathfrak j'=\mathfrak l=\mathfrak m_2,\qquad \mathfrak n=\mathfrak m_1.
\end{align*}
Given $T\in\mathcal M$, the equality
\begin{align*}
T=z_1\pi_{\mathfrak m_1}^*Q+z_2\pi_{\mathfrak m_2}^*Q
\end{align*}
holds for some $z_1,z_2>0$. 
It is obvious that 
\begin{align*}
\lambda_-(T|_{\mathfrak n})=z_1,\qquad \tr_Q T|_{\mathfrak l}=d_2z_2.
\end{align*}
A straightforward computation involving~\eqref{trace_bd} and~\eqref{Casimir} yields
\begin{align*}
\eta(\mathfrak g,\mathfrak s)&=\frac{2\tr_QB|_{\mathfrak m_1}+2\langle\mathfrak m_1\mathfrak m_2\mathfrak m_2\rangle+\langle\mathfrak m_1\mathfrak m_1\mathfrak m_1\rangle}{\omega(\mathfrak m_1)(2\tr_QB|_{\mathfrak m_2}+\langle\mathfrak m_2\mathfrak m_2\mathfrak m_2\rangle)} \\ &=\frac{2d_1b_1-2[122]-[111]}{d_1(2d_2b_2-[222])}
=\frac{4d_1\zeta_1+[111]}{d_1(4d_2\zeta_2+[222]+4[122])}.
\end{align*}
Theorem~\ref{thm_PRC} asserts that it is possible to find a metric $g\in\mathcal M$ whose Ricci curvature equals $cT$ for some $c>0$ if
\begin{align}
z_1/z_2>d_2\eta(\mathfrak g,\mathfrak s)=\frac{d_2(4d_1\zeta_1+[111])}{d_1(4d_2\zeta_2+[222]+4[122])}.
\end{align}
According to~\cite[Proposition~3.1]{AP16}, this condition is, in fact, sufficient \emph{and} necessary for the existence of~$g$.

\section{Generalised flag manifolds}\label{sec_flag}

In this section, we discuss the case where $M$ is a generalised flag manifold. Our first objective is to verify Hypothesis~\ref{hyp_flag}. After that, we will consider a class of examples to illustrate the use of Theorem~\ref{thm_PRC}. For the definition and some properties of a generalised flag manifold, see, e.g.,~\cite[Chapter~7]{AA03}. We will also rely on the classification results obtained in~\cite{MK90,SAIC11} and collected in~\cite{SAIC11}.

\begin{proposition}\label{prop_g_flag_mfs}
Suppose $M$ is a generalised flag manifold. Then $M$ satisfies Hypothesis~\ref{hyp_flag}.
\end{proposition}

\begin{proof}
Choose a decomposition of the form~\eqref{m_decomp}. Since $M$ is a generalised flag manifold, the representations $\Ad(H)|_{\mathfrak m_i}$ and $\Ad(H)|_{\mathfrak m_j}$ are inequivalent whenever~$i\ne j$; see, e.g.,~\cite[Chapter~7, Section~5]{AA03}. The summands $\mathfrak m_1,\ldots,\mathfrak m_s$ are determined uniquely up to order. Consequently, every nonzero $\Ad(H)$-invariant subspace of $\mathfrak m$ is the direct sum of $\mathfrak m_i$ with the index $i$ running through some non-empty subset of $\{1,\ldots,s\}$.

Let us verify Hypothesis~\ref{hyp_flag}. Consider a Lie subalgebra $\mathfrak s\subset\mathfrak g$ containing $\mathfrak h$ as a proper subset. It is obvious that $\mathfrak s\ominus\mathfrak h$ is $\Ad(H)$-invariant.
Therefore, for some $J_{\mathfrak s}\subset\{1,\ldots,s\}$,
\begin{align*}
\mathfrak s\ominus\mathfrak h=\bigoplus_{i\in J_{\mathfrak s}}\mathfrak m_i,\qquad \mathfrak g\ominus\mathfrak s=\bigoplus_{i\in\{1,\ldots,s\}\setminus J_{\mathfrak s}}\mathfrak m_i.
\end{align*}
As we noted above, $\Ad(H)|_{\mathfrak m_i}$ and $\Ad(H)|_{\mathfrak m_j}$ are inequivalent for~$i\ne j$. It follows that $\mathfrak s$ meets requirement~1 of Hypothesis~\ref{hyp_flag}.

As explained in~\cite[Chapter~7, Section~5]{AA03}, for every $i=1,\ldots,s$, the complexification of $\mathfrak m_i$ is the sum of two complex vector spaces of the same dimension. Consequently, $d_i$ is even. We conclude that $\mathfrak m$ does not have any $\Ad(H)$-invariant 1-dimensional subspaces. This means $\mathfrak s$ meets requirement~2 of Hypothesis~\ref{hyp_flag}.
\end{proof}

Let $M$ be a generalised flag manifold. Assume that $s=3$ in every decomposition of the form~\eqref{m_decomp} and that $M$ is of type~I in the terminology of~\cite{SAIC11}. Our next goal is to write down explicit formulas for the numbers $\eta(\mathfrak k,\mathfrak k')$ associated with simple chains of the form~\eqref{flag}. This will lead up to the application of Theorem~\ref{thm_PRC}. Analogous reasoning works if $M$ is of type~II in the terminology of~\cite{SAIC11} or if the isotropy representation of $M$ splits into four or five irreducible summands. We provide further details in Remark~\ref{rem_str_con_345} below.

Consider a decomposition
\begin{align*}
\mathfrak m=\mathfrak m_1\oplus\mathfrak m_2\oplus\mathfrak m_3
\end{align*}
of the form~\eqref{m_decomp}. It will be convenient for us to assume that this decomposition is the same as in~\cite[Subsection~2.4]{SAIC11}. The definition of a generalised flag manifold requires the group $G$ to be semisimple. This enables us to set $Q=-B$. According to~\cite[Formulas~(11), (13) and~(15)]{SAIC11},
\begin{align}\label{struc_const_flag3}
[112]&=[121]=[211]=\frac{d_1d_2+2d_1d_3-d_2d_3}{d_1+4d_2+9d_3}, \notag \\
[123]&=[231]=[312]=[321]=[213]=[132]=\frac{(d_1+d_2)d_3}{d_1+4d_2+9d_3},
\end{align}
and the rest of the structure constants are~0. The dimensions $d_1,d_2,d_3$ for concrete spaces are listed in~\cite[Table~4]{SAIC11}.

\begin{remark}\label{rem_str_con_345}
The reader will find the structure constants of generalised flag manifolds with two irreducible isotropy summands in~\cite{AAIC11,AA15}, three summands in~\cite{MK90,SAIC11}, four summands in~\cite{AAIC10} and five summands in~\cite{AAICYS13}.
\end{remark}

As we mentioned in the proof of Proposition~\ref{prop_g_flag_mfs}, the representations $\Ad(H)|_{\mathfrak m_i}$ and $\Ad(H)|_{\mathfrak m_j}$ are inequivalent for~$i\ne j$. Consequently, every nonzero $\Ad(H)$-invariant subspace of $\mathfrak g\ominus\mathfrak h$ is the direct sum of some of the spaces $\mathfrak m_1$, $\mathfrak m_2$ and $\mathfrak m_3$. This fact and formulas~\eqref{struc_const_flag3} imply that the proper Lie subalgebras of $\mathfrak g$ containing $\mathfrak h$ as a proper subset are 
\begin{align*}
\mathfrak s_1=\mathfrak m_2\oplus\mathfrak h,\qquad \mathfrak s_2=\mathfrak m_3\oplus\mathfrak h.
\end{align*}
It follows that the simple chains associated with $M$ are
\begin{align*}
\mathfrak g\supset\mathfrak g\supset\mathfrak s_1\supset\mathfrak h,\qquad \mathfrak g\supset\mathfrak g\supset\mathfrak s_2\supset\mathfrak h.
\end{align*}
Given $T\in\mathcal M$, the equality
\begin{align*}
T=-z_1\pi_{\mathfrak m_1}^*B-z_2\pi_{\mathfrak m_2}^*B-z_3\pi_{\mathfrak m_3}^*B
\end{align*}
holds for some $z_1,z_2,z_3>0$. Setting $\mathfrak k=\mathfrak g$ and $\mathfrak k'=\mathfrak s_i$ in~\eqref{flag}, we obtain
\begin{align*}
\mathfrak j=\{0\},\qquad \mathfrak j'=\mathfrak l&=\mathfrak m_1\oplus\mathfrak m_{4-i},\qquad \mathfrak n=\mathfrak m_{1+i}, \\
\lambda_-(T|_{\mathfrak n})&=z_{1+i},\qquad \tr_{-B} T|_{\mathfrak l}=d_1z_1+d_{4-i}z_{4-i},\qquad i=1,2.
\end{align*}
A computation involving~\eqref{trace_bd}, \eqref{<ijk>[ijk]} and~\eqref{struc_const_flag3} yields
\begin{align*}
\eta(\mathfrak g,\mathfrak s_1)&=\frac{2\tr_{-B}B|_{\mathfrak m_2}+2(\langle\mathfrak m_2\mathfrak m_1\mathfrak m_1\rangle+\langle\mathfrak m_2\mathfrak m_3\mathfrak m_3\rangle+2\langle\mathfrak m_2\mathfrak m_1\mathfrak m_3\rangle)+\langle\mathfrak m_2\mathfrak m_2\mathfrak m_2\rangle}{\omega(\mathfrak m_2)(2\tr_{-B}B|_{\mathfrak m_1\oplus\mathfrak m_3}+\langle\mathfrak m_1\mathfrak m_1\mathfrak m_1\rangle+\langle\mathfrak m_3\mathfrak m_3\mathfrak m_3\rangle+3\langle\mathfrak m_1\mathfrak m_1\mathfrak m_3\rangle+3\langle\mathfrak m_1\mathfrak m_3\mathfrak m_3\rangle)}
\\
&=\frac{-d_2+[112]+2[123]}{d_2(-d_1-d_3)}
=\frac{-4d_2^2-8d_2d_3+4d_1d_3}{-d_2(d_1+d_3)(d_1+4d_2+9d_3)},
\\
\eta(\mathfrak g,\mathfrak s_2)&=\frac{2\tr_{-B}B|_{\mathfrak m_3}+2(\langle\mathfrak m_3\mathfrak m_1\mathfrak m_1\rangle+\langle\mathfrak m_3\mathfrak m_2\mathfrak m_2\rangle+2\langle\mathfrak m_3\mathfrak m_1\mathfrak m_2\rangle)+\langle\mathfrak m_3\mathfrak m_3\mathfrak m_3\rangle}{\omega(\mathfrak m_3)(2\tr_{-B}B|_{\mathfrak m_1\oplus\mathfrak m_2}+\langle\mathfrak m_1\mathfrak m_1\mathfrak m_1\rangle+\langle\mathfrak m_2\mathfrak m_2\mathfrak m_2\rangle+3\langle\mathfrak m_1\mathfrak m_1\mathfrak m_2\rangle+3\langle\mathfrak m_1\mathfrak m_2\mathfrak m_2\rangle)}
\\ 
&=\frac{-2d_3+4[123]}{d_3(-2d_1-2d_2+3[112])} 
=\frac{-2d_1+4d_2+18d_3}{2d_1^2+8d_2^2+7d_1d_2+12d_1d_3+21d_2d_3}.
\end{align*}
Theorem~\ref{thm_PRC} tells us that a Riemannian metric with Ricci curvature equal to $cT$ for some $c>0$ exists if 
\begin{align*}
\frac{z_2}{d_1z_1+d_3z_3}&>\frac{-4d_2^2-8d_2d_3+4d_1d_3}{-d_2(d_1+d_3)(d_1+4d_2+9d_3)},
\\
\frac{z_3}{d_1z_1+d_2z_2}&>\frac{-2d_1+4d_2+18d_3}{2d_1^2+8d_2^2+7d_1d_2+12d_1d_3+21d_2d_3}.
\end{align*}

\begin{example}
Suppose $M$ is the generalised flag manifold $G_2/U(2)$ in which $U(2)$ corresponds to the long root of~$G_2$. According to~\cite[Table~4]{SAIC11}, in this case, $d_1=d_3=4$ and $d_2=2$. Theorem~\ref{thm_PRC} implies that a Riemannian metric with Ricci curvature equal to $cT$ for some $c>0$ exists if 
\begin{align*}
\frac{z_2}{z_1+z_3}>\frac1{12},\qquad
\frac{z_3}{2z_1+z_2}>\frac3{10}.
\end{align*}
\end{example}

\section{Ricci iterations}\label{sec_iter}

Corollary~\ref{cor_all} yields a new existence result for Ricci iterations on homogeneous spaces. More precisely, the following proposition holds. For earlier work on the subject, see~\cite{APYRsubm}.

\begin{proposition}\label{thm_iter}
Suppose the homogeneous space $M$ admits a decomposition of the form~(\ref{m_decomp}) satisfying requirements~1 and~2 of Corollary~\ref{cor_all}. Given a metric $\bar g_1\in\mathcal M$, there exists a sequence $(g_i)_{i=1}^\infty\subset\mathcal M$ such that the formulas $g_1=c_1\bar g_1$ and~(\ref{eq_iter}) hold for some $c_1>0$ and all~$i\in\mathbb N\setminus\{1\}$.
\end{proposition}

\begin{proof}
We construct $(g_i)_{i=1}^\infty$ inductively. An application of Corollary~\ref{cor_all} with $T=\bar g_1$ yields the existence of $\bar g_2\in\mathcal M$ such that 
\begin{align*}
\Ric\bar g_2=c_1\bar g_1
\end{align*} 
for some $c_1>0$. We set $g_1=c_1\bar g_1$. One more application of Corollary~\ref{cor_all}, this time with $T=\bar g_2$, produces $\bar g_3\in\mathcal M$ such that 
\begin{align*}
\Ric\bar g_3=c_2\bar g_2
\end{align*} 
for some $c_2>0$. We set $g_2=c_2\bar g_2$. It is obvious that $\Ric g_2$ coincides with $g_1$. Continuing in this way, we obtain $(g_i)_{i=1}^\infty$.
\end{proof}

We discussed several examples of $M$ satisfying the assumptions of Proposition~\ref{thm_iter} in the end of Subsection~\ref{subsec_cor}. For a detailed description of the behaviour of Ricci iterations on homogeneous spaces with two inequivalent irreducible isotropy summands, see~\cite[Theorem~2.1]{APYRsubm}.

\section*{Acknowledgements}

We express our gratitude to Ole Warnaar for helpful discussions.

\end{document}